\pdfmapfile{=<name>.map}
\documentclass[12pt,oneside,reqno,fleqn]{amsart}
\usepackage{bbm}
\usepackage{cases}
\usepackage{graphicx}
\usepackage{mathrsfs}
\usepackage{stmaryrd}
\usepackage{color}
\usepackage[dvipsnames]{xcolor}
\usepackage{amsfonts}
\usepackage{amsmath,amssymb,amsthm,mathtools}
\usepackage{libertine}
\usepackage[T1]{fontenc}
\usepackage[libertine,varbb]{newtxmath}
\usepackage[colorlinks,linkcolor=Red,citecolor=Red,urlcolor=Red]{hyperref}
\usepackage[shortlabels] {enumitem}
\usepackage[nameinlink,capitalize]{cleveref}
\usepackage[spacing=true,kerning=true,tracking=true,letterspace=50]{microtype}
\usepackage{comment}
\hypersetup{pdftitle={Path-by-path uniqueness for stochastic differential equations under Krylov--Röckner condition}, pdfauthor={Lukas Anzeletti, Khoa L\^e and Chengcheng Ling}}
\numberwithin{equation}{section}
\usepackage[
disable
]{todonotes}

\newcommand{\be}{\begin{eqnarray}}
\newcommand{\ee}{\end{eqnarray}}
\newcommand{\ce}{\begin{eqnarray*}}
\newcommand{\de}{\end{eqnarray*}}
\newtheorem{theorem}{Theorem}[section]
\newtheorem{lemma}[theorem]{Lemma}
\newtheorem{proposition}[theorem]{Proposition}
\newtheorem{corollary}[theorem]{Corollary}
\theoremstyle{remark}
\newtheorem{assumption}[theorem]{Assumption}

\newtheorem{example}[theorem]{Example}
\newtheorem{remark}[theorem]{Remark}
\newtheorem{definition}[theorem]{Definition}

\crefname{eqn}{Equation}{Equations}
\crefname{assumption}{Assumption}{Assumptions}
\crefname{innercustomthm}{Condition}{Conditions}
\crefrangelabelformat{innercustomthm}{#3#1#4-#5#2#6}

\def\<{{\langle}}
\def\>{{\rangle}}
\def\({{\Big(}}
\def\){{\Big)}}

\def\bx{{\mathbf{x}}}

\def\min{{\mathord{{\rm min}}}}

\def\={&\!\!=\!\!&}
\def\bt{\begin{theorem}}
\def\et{\end{theorem}}
\def\bl{\begin{lemma}}
\def\el{\end{lemma}}
\def\br{\begin{remark}}
\def\er{\end{remark}}
\def\bd{\begin{definition}}
\def\ed{\end{definition}}
\def\bp{\begin{proposition}}
\def\ep{\end{proposition}}
\def\bc{\begin{corollary}}
\def\ec{\end{corollary}}
\def\bx{\begin{example}}
\def\ex{\end{example}}

\def\cF{{\mathcal F}}
\def\cff{\cF}

\def\mE{{\mathbb E}}
\def\E{\mE}

\def\mP{{\mathbb P}}

\def\mR{{\mathbb R}}

\def\geq{\geqslant}
\def\leq{\leqslant}
\def\les{\lesssim}

\newcommand{\dd}{\,\mathrm{d}}
\newcommand{\dr}{\dd r}
\newcommand{\R}{{\mathbb R}}
\newcommand{\Rd}{{\R^d}}
\newcommand{\tand}{\quad\text{and}\quad}
\newcommand{\var}{\mathrm{var}}
\newcommand{\LL}{{\mathbb{L}}}

\newcommand{\1}{{\mathbf 1}}

\newcommand{\cjj}{{\mathcal J}}
\newcommand\abz{{\mathcal Z}}

\newcommand{\norm}[1]{{\left\vert\kern-0.25ex\left\vert\kern-0.25ex\left\vert #1
    \right\vert\kern-0.25ex\right\vert\kern-0.25ex\right\vert}}

\renewcommand{\le}{\leq}
\renewcommand{\ge}{\geq}
\allowdisplaybreaks

\newcommand{\simpl}[2]{[#1,#2]^2_{\le}}
\newcommand{\simt}{\simpl{0}{T}}
\newcommand{\simst}{\simpl{S}{T}}
\newcommand{\simi}{I^2_\leq}
\newcommand{\siminfi}{[0,\infty)^2_{\le}}

\allowdisplaybreaks

\begin{document}
	\title{Path-by-path uniqueness for stochastic differential equations under Krylov--Röckner condition}
	
	\date{\today}
	
\author{Lukas Anzeletti, Khoa L\^e and Chengcheng Ling}

\address{Lukas Anzeletti:
Technische Universit\"at Wien,
Institute of Analysis and Scientific Computing,
1040 Wien, Austria
\\
Email: lukas.anzeletti@tuwien.ac.at}
\address{Khoa L\^e: University of Leeds,
School of Mathematics,  Leeds 
LS29JT, U.K.
\\
Email: k.le@leeds.ac.uk}

\address{Chengcheng Ling:
Universit\"at Augsburg,
Institute of Mathematics,
86159 Augsburg, Germany
\\
Email: chengcheng.ling@uni-a.de
 }

	\begin{abstract}
We show that any stochastic differential equation (SDE) driven by Brownian motion with drift satisfying the Krylov--R\"ockner condition has exactly one solution in an ordinary sense for almost every trajectory of the Brownian motion. 
Consequentially, such SDE is strongly complete and forms a random dynamical system. Also, a further application to a boundary value problem is discussed.
	
		\bigskip
		
		
	\end{abstract}
	
	\maketitle

\section{Introduction}
For an integer dimension $d\geq 1$, a measurable time-dependent vector field $b:[0,\infty)\times \mathbb{R}^d \rightarrow \mathbb{R}^d$ and a measurable  driving signal $\gamma: [0,\infty)\rightarrow \mathbb{R}^d$ with $\gamma_0=0$, we consider the deterministic ordinary differential equation (ODE)
\begin{equation}\label{ODE}
    \dd y_t=b(t,y_t)\dd t+\dd \gamma_t.
\end{equation} 
A solution to \eqref{ODE} over a finite time period $[S,T]$ starting at $x\in\Rd$ is a measurable function $(y_t)_{t\in[S,T]}$ which satisfies
\begin{align}\label{intro.intb}
    \int_S^T|b(r,y_r)|\dd r<\infty,
\end{align}
and 
\begin{equation}\label{ODEint}
y_t=x+\int_S^t b(r,y_r) \dd r + (\gamma_t-\gamma_S), \quad  {t \in [S,T]}.
\end{equation}
The integral in the above formulation  is well-defined in Lebesgue sense and will be understood in such sense throughout. 
For such solution, it is evident that $y- \gamma$ is continuous.

    When $\gamma$ is sampled from a stochastic process $W$ on a stochastic basis $(\Omega,\cff,\mP)$, equation \eqref{ODE} is related to the stochastic differential equation (SDE)
    \begin{align}\label{SDE}
        \dd Y_t=b(t,Y_t)\dd t+\dd W_t.
    \end{align}
    A stochastic process $Y:\Omega\times[S,T]\to\Rd$ is a solution to \eqref{SDE} over a finite time period $[S,T] $ starting at $x\in\Rd$ if for a.s. $\omega$, $(Y_t(\omega))_{t\in[S,T]}$ is a solution to \eqref{ODE} with $W(\omega)$ in place of $\gamma$, 
    i.e. to the $W(\omega)$-driven ODE
    \begin{align}\label{eqn.ODEomega}
        \dd Y_t= b(t,Y_t)\dd t+\dd W_t(\omega).
    \end{align}

    It is now well-understood that the addition of fast oscillating driving signals $\gamma, W$ can make both \eqref{ODE} and \eqref{SDE} well-posed for very irregular vector fields $b$, which is known as \emph{regularization-by-noise} phenomenon.

    For SDEs driven by Brownian motion with bounded measurable vector fields, \cite{Z,Ve} showed that there exists a pathwise unique strong solution to \eqref{SDE}. In \cite{Kr-Ro}, Krylov and R\"ockner  extended this result to drifts $b$ satisfying the condition
    \begin{align}\label{LpLq}
        \Big(\int_0^T\big(\int_\Rd|b(r,x)|^p\dd x\big)^{\frac qp}\dd r\Big)^{\frac1q}<\infty,\quad  (p,q)\in
    \mathcal{J}\coloneqq\left\{(p,q)\in (2,\infty)^2: \frac{2}{q}+\frac{d}{p}<1\right\}.
    \end{align}
    This condition is also known as the \emph{sub-critical Ladyzhenskaya--Prodi--Serrin} (LPS) condition in fluid dynamics, see more in \cite{FF12, BeckFlandoli}.
    Pathwise uniqueness for SDEs with fractional noise was considered first by Nualart and Ouknine in \cite{MR1934157,MR2073441} and more recently in \cite{banos2015strong,le2020stochastic,anzeletti2021regularisation,GaleatiGerencser,butkovsky2023stochastic} among others.

    Well-posedness for \eqref{ODE} is considered in more recent time. Starting from the seminal works of Davie \cite{Davie} and Catellier--Gubinelli \cite{CatellierGubinelli}, there has been an increasing interest in the problem, \cite{AVS,HarangPerkowski,GaleatiGubinelli,GaleatiGerencser, DG22}. A typical result of these articles is that when $\gamma$ is a sample path of an irregular stochastic process (such as fractional Brownian motion), then for almost every such sample paths, ODE \eqref{ODE} has a unique solution, a property known by the terminology \emph{path-by-path uniqueness}.
    The necessary regularity for the vector field is inversely proportional to the roughness of the stochastic process from which $\gamma$ is sampled. 
    In a different direction, when the coefficients are sufficiently regular so that Lyons' rough path theory can be applied, then uniqueness of the corresponding rough differential equation implies path-by-path uniqueness. This is because the stochastic basis is only used to construct a rough path lift of the driving signal and Lyons' rough path analysis (\cite{MR1654527}) is applicable for any fixed trajectory of the lift.
    It is important to highlight that,  according to \cite{Davie,CatellierGubinelli}, the null events are dependent on the initial data. In contrast,  the rough path methodology  implies that the null events are not influenced by the initial data. To differentiate the two properties, we will refer to the latter as \textit{uniform path-by-path uniqueness}, see \cref{def.unipbp} below for a precise statement.

    From a practical perspective, path-by-path uniqueness is most relevant in situations where only a handful samples of the signal can be observed.
    Theoretically, ``path-by-path uniqueness'' is stronger than ``pathwise uniqueness'' (which essentially allows to identify two adapted solutions to \eqref{SDE}).
    Having pathwise uniqueness for \eqref{SDE}
    does not provide uniqueness for the $W(\omega)$-driven ODE \eqref{eqn.ODEomega}
    for a given $\omega$. In fact, examples of SDEs driven by Brownian motion which have pathwise uniqueness but not path-by-path uniqueness are given in \cite{ShaWre,Anzeletti}. We point to \cref{pbpvspathwise} at the end of the current article for further discussions on the related concepts. 

    The main goal of the current article is the following result, which will be stated more precisely later in \cref{thm.main}. %

    \smallskip
    \textbf{\textsc{Theorem.}}  \textit{Assuming that $W$ is a standard Brownian motion, then uniform path-by-path uniqueness
    and strong completeness for \eqref{SDE} holds under the Krylov--R\"ockner condition \eqref{LpLq}.} %
    \smallskip

In this context,   \emph{strong completeness} refers to the property that  for almost all trajectories of the Brownian motion, the equation generates a random semiflow over all nonnegative time.
    Strong completeness is a rather demanding property and may fail even for SDEs with bounded smooth  coefficients (but with unbounded derivatives), with examples given by Li and Scheutzow in \cite{LS11}. 
    Strong completeness for SDEs driven by Brownian motion with standard regularity conditions are considered in \cite{MR1305784,SS17}.
    Alternatively, when the SDE under consideration can be lifted to a rough differential equation which has unique global solution, then it is strongly complete, see \cite{MR3567487}. 
    The SDE \eqref{SDE} under condition \eqref{LpLq} falls outside the standard It\^o framework and Lyons' rough path theory. Hence, our result provides the first example of strongly complete singular SDEs driven by Brownian motion.

    \smallskip
 { Following the establishment of \emph{path-by-path uniqueness} and \emph{strong completeness} under the Krylov--Röckner condition, we also consider their applications, which are highlighted below and detailed in \cref{sec:implications}.
}{
 
       \begin{enumerate}
    \item \emph{Inverse flow}\\
        The inverse of the solution flow (\emph{inverse flow}) to an SDE is useful in wider contexts, e.g. representation of solutions to the singular stochastic transport equation  \cite{MR2593276, MR3017266}, and to nonlinear partial differential equations \cite{MR2376844, MR2653231}. 
        Because of irregularity of the drift and the   dependence of stochastic analysis on the direction of time,
        construction of the inverse flow often leads to many technical difficulties (\cite{MR1472487,FF}).
        Using \emph{path-by-path uniqueness}, we demonstrate in \cref{cor:flow-back} that the inverse flow for singular SDEs can be easily deduced as in classical theory for ODEs, in contrast with earlier approaches.
 
     \item \emph{Random dynamical systems (RDS's) induced by singular  SDEs}\\
     Over recent decades, RDSs, introduced by L. Arnold \cite{MR1723992}, have been extensively studied for understanding complex systems, among which many are described by SDEs. Herein, we explain that  SDEs with  singular drifts, such as \eqref{SDE} under \eqref{LpLq}, also generate RDS's.
     Traditionally, this is done by %
 perfection procedures \cite{MR1314175} %
 but this is a not simple task -- especially for singular SDEs due to regularity issues (e.g. \cite{MR1723992, MR1373727, MR4431447}). Path-by-path uniqueness, however, resolves this challenge naturally as it inherently implies that the underlying SDE is a well-posed random ODE, and therefore satisfies the RDS requirements without the need for such refinements, see \cref{RDS:path-by-path}.
  \item \emph{Boundary value problem}\\
  Consider the boundary value problem defined by the dynamic \eqref{SDE} with a constraint $F_0Y_0+F_1Y_1=0$ between the initial condition and terminal value over a bounded interval $[0,1]$.
Historically, boundary value problems for SDEs, which examine existence and uniqueness \cite{NualartPardoux, MR1002898}, absolute continuity of laws \cite{NualartPardoux},  Markovian properties \cite{NualartPardoux, MR1002898, MR1349171, MR580907}, and more recently even  for rough differential equations \cite{lejay:hal-03626402},  have been considered under certain continuity assumptions on the drifts.
In \cref{subsec:boundary}, we consider the well-posedness of these problems for SDEs with singular drifts. Path-by-path uniqueness plays a key role in reducing to a deterministic analysis, avoiding the need for any stochastic analysis as employed in earlier literature.
This approach is similar to that of \cite{lejay:hal-03626402} in which the authors deal with problems with regular coefficients. Hence, showing that this approach also works with problems corresponding to singular drifts under condition \eqref{LpLq} is our novel contribution.

    \end{enumerate}}
    Through the applications described above, we hope to convey that path-by-path uniqueness, once established, is a useful essential property. On one hand, it overcomes earlier technical difficulties, on the other hand, it opens up pathways to new and nontrivial implications.
   
    \smallskip
    \noindent\textbf{Discussion on the literature.} 
    In his groundbreaking work \cite{Davie}, Davie established the first result on path-by-path uniqueness for the SDE \eqref{SDE}  with bounded measurable drift. 
    His argument can be loosely summarized in three main steps. First, Davie provided the regularizing estimates (see \cite[Proposition 2.1]{Davie})
    \begin{align}\label{intro.Davie1}
    \left\|\int_0^1[g(r,W_r+x)-g(r,W_r)]\dd r \right\|_{L^m(\Omega)}\le C(m/2)!
    |x|,
    \end{align}
    where $m\ge2$, $g:[0,1]\times \Rd\to\R$ is a Borel function bounded by $1$, $x\in\Rd$ and $W$ is a standard Brownian motion. In the second step, \eqref{intro.Davie1} was enhanced to $\omega$-wise estimates for the map 
    \begin{align}\label{intro.Davie2}
     t\mapsto\int_0^tg(r,W_r(\omega)+\psi_r)\dd r
     \end{align}
    for every Lipschitz function $\psi$ (compare with \cite[Lemma 3.6]{Davie}). In the last step, these estimates were applied to the solution of \eqref{SDE} (as the integral part of the solution is Lipschitz due to boundedness of $b$) to derive path-by-path uniqueness (see \cite[Lemma 3.7]{Davie}).

    Davie's approach, rooted in first  principles, unfortunately added complexity, especially in the final step. Nevertheless, his work has significantly influenced subsequent research, inspiring various new directions. For instance, Catellier and Gubinelli in \cite{CatellierGubinelli} reinterpreted and expanded upon Davie's estimates within the context of nonlinear Young integration. To see this connection, one should compare \cite[Outline of proof, pg. 12]{Davie} with \cite[begining of Section 2]{CatellierGubinelli}, and with \cref{lem.nonlinYoung} herein. By leveraging higher regularity in Davie's estimates, in exchange for asking for more regularity of the drift or more regularizing effect from the noise, \cite{CatellierGubinelli} obtained path-by-path uniqueness for a large class of fractional SDEs, although not encompassing the case considered by Davie.  
    Another notable work is \cite{AVS} by  Shaposhnikov. 
    Therein, the author gave another proof for \eqref{intro.Davie1} using It\^o calculus, new estimates for \eqref{intro.Davie2} by means of entropic numbers of space of Lipschitz paths (\cite[Proposition 2.1, Lemma 3.3]{AVS}), and a novel way of proving path-by-path uniqueness    utilizing these estimates and the semiflow generated by \eqref{SDE}, see \cite[Proposition 2.1, Lemma 3.3 and proof of Theorem 1.1. respectively]{AVS}. Therefore,  Shaposhnikov successfully recovered Davie's result by simpler justifications for Davie's three main steps.  
    In more recent works, techniques from \cite{CatellierGubinelli,AVS} have been combined where Davie-type estimates were obtained through nonlinear Young integration (or sewing techniques) and path-by-path uniqueness was achieved through Shaposhnikov's argument utilizing properties of the generated semiflow. For instance, \cite{GaleatiGerencser,DG22} showed path-by-path uniqueness for a large class of fractional SDEs with time-inhomogeneous drifts.
    However, it is intriguing that most recent advancements since \cite{AVS} have not been able to replicate Davie's result, highlighting the ongoing challenges in achieving path-by-path uniqueness for SDEs with irregular drifts.
    This underscores the complexity of the problem for which  the current article aims to resolve.

    Path-by-path uniqueness for \eqref{SDE} under condition \eqref{LpLq} has been speculated and attempted in \cite{AVS,F}. \cite{F} stopped shortly after deriving Davie's estimate \eqref{intro.Davie1} under condition \eqref{LpLq} with a non-optimal constant, which according to \cite{Flandoli}, ``do[es] not allow to apply the arguments of the second part of the proof of'' \cite{Davie}.
    In \cite[Remark 3.8]{AVS}, the author pointed out two major obstacles. One is showing Davie's estimate \eqref{intro.Davie1} under \eqref{LpLq} with a quantitative constant on the right-hand side. The other one is that the integral part of the solution to the $W(\omega)$-driven ODE \eqref{eqn.ODEomega}  is no longer Lipschitz due to the unboundedness of $b$ but only expected to be H\"older continuous (as if it was the same as the strong solution). 
    However, given sharp Davie estimates (which we establish in \eqref{prop.davie} below), solely employing Shaposhnikov's approach, i.e. utilizing estimates for entropy numbers of space of H\"older paths, would lead to a more restricted condition compared to \eqref{LpLq}.
    Alternatively, one could arguably measure the regularity of the integral part using variation, however the space of paths of finite variations does not have finite entropy numbers.
    Lastly, within the nonlinear Young framework, such as \cite{GaleatiGerencser,CatellierGubinelli,DG22}, the standing assumptions are stricter and do not allow condition \eqref{LpLq}.
    To summarize, the existing approaches from literatures can not be carried forward to treat SDEs under condition \eqref{LpLq}.

    The current paper provides a comprehensive approach which is rather general and could be applied to other classes of equations.
    First,  we obtain sharp Davie estimates under the Krylov-R\"ockner condition \eqref{LpLq} (see \cref{prop.davie}) by means of John--Nirenberg inequality, which effectively reduces estimating all positive moments to estimating only the second moment. In fact, contrary to previous approaches, sharp quantitative constants for Davie estimates are not essential for our method. Nevertheless, usage of John--Nirenberg shortens existing arguments and, at the same time, refines known results from \cite{Davie,rezakhanlou2014regular,F} by providing sharp quantitative constants. Building upon Davie estimates, we  use nonlinear Young integration to obtain a key regularizing estimate for averaging operators \eqref{intro.Davie2} along Brownian motion, namely \eqref{est.young} herein.
    This development is particularly significant as it leverages the finite variation property of the integral part of the solution, which sidesteps the need for entropy estimates, as employed in \cite{AVS}. Consequently, this allows us to show that equation \eqref{SDE} generates a random continuous semiflow over finite time periods and to obtain  a priori estimates for any solution to \eqref{eqn.ODEomega} (\cref{prop.apri1,prop.regularityabs}). Applying these a priori estimates and the regularity of the semiflow generated by \eqref{SDE}, we are able to identify any solution with the semiflow, thereby, showing path-by-path uniqueness. Strong completeness is an immediate consequence of path-by-path uniqueness and continuity of the random semiflow over arbitrary finite time periods. 

    Although our approach employs  similar tools as those used in \cite{CatellierGubinelli,AVS,GaleatiGerencser}, it is both more effective and logically distinct from the aforementioned works. For instance, due to discontinuity of $b$, showing that \eqref{SDE} generates a random continuous semiflow is subtly nontrivial and has been overlooked in previous works. Our approach can recover the case considered by Davie while at the same time, circumvent the major obstacles pointed out by Shaposhnikov. 
    Additionally, our arguments are generic and are not strictly tied to Brownian motion, as long as the random semiflow generated by the SDE \eqref{SDE} and the averaging operators along the driving signal $W$ are sufficiently regular.
    Consequentially, this leads to a new uniqueness criterion for ODEs of the form \eqref{ODE},
    which does not require any regularity on the vector field, but instead relies on the regularizing effect of the driving signal $\gamma$. These results are summarized in \cref{sec:ode_uniqueness}.

    \smallskip
    \noindent\textbf{Structure of paper.} Main results are stated in \cref{sec:main_results}. In \cref{sec.Davie}, we show Davie estimates under \eqref{LpLq}. 
    In \cref{sec:path_estimates}, we show some path-by-path regularizing estimates, which are consequences of Davie estimates and the sewing lemma.
    The proof of the main results are presented in \cref{sec:proof_main_results}. 
    The applications and proofs are presented in \cref{sec:implications}. In \cref{sec:ode_uniqueness}, we summarize a uniqueness criterion for \eqref{ODE} which is based on the regularization effect of the driving signal.
    Different notions of solutions and their relations are recalled in \cref{pbpvspathwise}. 

    \smallskip
    \noindent\textbf{Frequently used notation.}
   Define for any random variable $X$ and  for any $m\geq1$, $ \Vert X\Vert_{L^m(\Omega)}\coloneqq[\mE(|X|^m)]^{1/m}$. 
   For any $R>0$, we denote by $B_R$ the closed ball in $\R^d$ with radius $R$ centered at the origin.
    For each $p\in[1,\infty]$, $L^p(\Rd)$ denotes the usual Lebesgue space on $\mR^d$.
    For any $0\leq s\leq t$, the space $\LL_p^q([s,t])$ contains all measurable functions $f:[s,t]\times\mR^d\rightarrow\mR$ such that 
    \begin{align*}
       \Vert f\Vert_{\LL_p^q([s,t])}\coloneqq \Big(\int_s^t\big(\int_{\mR^d}|f(r,x)|^p\dd x\big)^{\frac{q}{p}}\dd r\Big)^{\frac{1}{q}}<\infty
    \end{align*}
    for $p,q\in[1,\infty)$ and with standard modifications when $p$ or $q$ is infinite. We abbreviate $\LL_p^q=\LL_p^q(\R)$.
    For $I\subset [0,\infty)$, let $I^2_\leq\coloneqq\{(s,t)\in I\times I:s,t \in I; s\leq t\}$.
    We say that a continuous function $w:\simpl{S}{T}\to [0,\infty)$ is a control if $w(s,u)+w(u,t)\leq w(s,t)$ whenever $s\le u\le t$.
    For a function $(\psi_t)_{t\in[0,T]}$ of finite variation and each $(u,v)\in[0,T]_\leq^2$, we define 
    \begin{align*}%
        [\psi]_{\var;[u,v]}\coloneqq\sup_{\pi}\sum_{[s,t]\in\pi}|\psi_t-\psi_s|,
    \end{align*}
      where the supremum is taken over all partitions $\pi$ of $[u,v]$. It is well documented that $(u,v)\mapsto [\psi]_{\var;[u,v]}$ is a control, see \cite{friz2010multidimensional}. 
    We write $H\les G$ if $H\le C G$ for some universal finite positive constant $C$.

    Let $\{\cff_t\}_{t\ge0}$  be a filtration and assume that $(\Omega,\cff,\{\cff_t\}_{t\ge0}, \mP)$ satisfies the usual conditions. 
    For each $s\in[0,T]$, $\mE_s$ denotes the conditional expectation with respect to $\cff_s$, i.e. $\E_s[\cdot]=\E[\,\cdot\,|\cff_s]$.
    We assume  that $W$ is a standard Brownian motion with respect to $\{\cff_t\}_{t\ge0}$.
    \newpage
    \section{Main results} %
\label{sec:main_results}

    To state our results precisely, we briefly recall some key concepts.

   \begin{definition}
       \label{def:randomsemiflow}
        Let $I=[0,T]$ for some finite constant $T>0$ or $I=[0,\infty)$.   
    A \emph{semiflow} generated by the ODE \eqref{ODE} over the period $I$ is a measurable map $(s,t,x)\mapsto\phi^{s,x}_t$ defined on $\simi\times\Rd$ such that for every $s,u,t\in I$,  $0\le s\le u\le t$, and every $x\in\Rd$,
    \begin{gather*}
        \int_s^t |b(r,\phi^{s,x}_r)|\dd r<\infty, \quad\phi^{s,x}_t=x+\int_s^t b(r,\phi^{s,x}_r)\dd r+ \gamma_t-\gamma_s
        \tand\phi^{s,x}_t=\phi^{u,\phi^{s,x}_u}_t.
    \end{gather*}
    For a semiflow $\phi$ on $ I$, we say that $\phi$ is continuous if $(s,t,x)\mapsto \phi_t^{s,x}$ is continuous on $\simi\times\Rd$.   
    We say that $\phi$ is \emph{locally $\kappa$-H\"older continuous in space} if 
    for every compact set  $[u,v]\times K\subset I\times\Rd$, there exists a finite constant $C=C(\kappa,[u,v]\times K)$ such that for every $(s,t) \in [u,v]^2_\le$ and $x,y \in K$
    \begin{align*}
        |\phi^{s,x}_t- \phi^{s,y}_t|\le C|x-y|^\kappa.
    \end{align*}
   \end{definition}

    Next, we give precise definitions of path-by-path uniqueness, random semiflow and strong completeness for \eqref{SDE}.
    \begin{definition}\label{def.unipbp}
        Let $T>0$ and $S\in[0,T)$. We say that \emph{path-by-path uniqueness} for \eqref{SDE} holds over $[S,T]$  if for  any $x\in \Rd$ and any probability space on which a Brownian motion $W$ is defined, there exists a set of full measure $\tilde{\Omega}$, such that for each $\omega \in \tilde{\Omega}$, there exists a unique solution on $[S,T]$ starting at $x$  to the $W(\omega)$-driven ODE \eqref{eqn.ODEomega}.

        We call such property uniform if the set $\tilde{\Omega}$ is independent from $S,T,x$. In other words, we say that \emph{uniform path-by-path uniqueness} for \eqref{SDE} holds if for any probability space on which a Brownian motion $W$ is defined, there exists a set of full measure $\tilde{\Omega}$, such that for all $0\le S<T$, $x \in \mathbb{R}^d$ and any $\omega \in \tilde{\Omega}$, there exists a unique solution on $[S,T]$ starting at $x$  to the $W(\omega)$-driven ODE \eqref{eqn.ODEomega}.
        \end{definition} 
    \begin{definition} \label{def:semiflow}
        We say that \eqref{SDE} generates a random semiflow $X=(X^{s,x}_t)_{s,t,x}$ over the period $I$ if the map
        \begin{align*}
            \simi\times\Rd\times \Omega&\to\Rd
             \\(s,t,x,\omega)&\mapsto X^{s,x}_t(\omega)
        \end{align*} 
        is measurable and for a.s. $\omega$, $(X^{s,x}_t(\omega))_{s,t,x}$ is a semiflow generated by the ODE \eqref{eqn.ODEomega} over $I$.
        
        For a random semiflow $X$ on $ I$, we say that $X$ is continuous
        if for a.s. $\omega$, the map  $(s,t,x)\mapsto X_t^{s,x}(\omega)$ is continuous on $\simi\times\Rd$. We
        say that $X$ is \emph{locally $\kappa$-H\"older continuous in space}
        if for a.s. $\omega$, $(X^{s,x}_t(\omega))_{s,t,x}$ is locally $\kappa$-H\"older continuous in space (see Definition~\ref{def:randomsemiflow}).
    \end{definition}
    \begin{definition}\label{def.strcomplete}
        We say that the SDE \eqref{SDE} is \emph{strongly complete} if for a.s. $\omega$, the ODE \eqref{eqn.ODEomega} generates a continuous semiflow over $[0,\infty)$. This means that there exists a measurable map 
        \begin{align*}
            \siminfi\times\Rd\times \Omega&\to\Rd
             \\(s,t,x,\omega)&\mapsto X^{s,x}_t(\omega)
         \end{align*}
        such that for a.s. $\omega$ the following properties hold:
        \begin{enumerate}[(i)]
            \item the map $(s,t,x)\mapsto X^{s,x}_t(\omega)$ is continuous;
            \item \label{it.strcompeqn} for every $0\le s\le u\le t$ and $x\in\Rd$,
            \begin{align*}
                \int_s^t|b(r,X^{s,x}_r(\omega))|\dr<\infty,
                \quad X^{s,x}_t(\omega)=x+\int_s^t b(r,X^{s,x}_r(\omega))\dr+W_t(\omega)-W_s(\omega)
            \end{align*}
            and
            \begin{align*}
                X^{s,x}_t(\omega)=X^{u,X^{s,x}_u(\omega)}_t(\omega).
            \end{align*}
        \end{enumerate}
    \end{definition}
        Our definition of strong completeness differs from the literatures' in that \ref{it.strcompeqn} demands that the event on which the semiflow solves the equation is independent from initial data.
        Heuristically, strong completeness requires that the random semiflow is defined  globally over all nonnegative times $[0,\infty)$ and continuously in $(s,t,x)$ over $\siminfi\times\Rd$. That the semiflow is generated by the equation means that for almost every $\omega$ and every initial data $(s,x)$, the map $t\mapsto X^{s,x}_t(\omega)$ is a global solution which does not blow up in finite time. 
        The latter property can be thought as ``path-by-path non-explosion''. 
        SDEs which have global pathwise solutions without explosion may fail to have path-by-path non-explosion with examples given in \cite{LS11}.

    Our main results are stated in the following theorem.
    \begin{theorem}\label{thm.main}
    Assume that $b$ fulfills \eqref{LpLq} for all $T<\infty$. There exists an event $\Omega_{b}\in\cff$ which depends only on $b$ and has full probability measure such that for every $\omega\in \Omega_{b}$, the following statements hold true.
    \begin{enumerate}[(A)]
        \item\label{it.complete} {[Strong completeness]} The $W(\omega)$-driven  ODE \eqref{eqn.ODEomega} generates a continuous semiflow $(X^{s,x}_t(\omega))_{s,t,x}$ over $[0,\infty)$   which is locally $\kappa$-H\"older continuous in space for every $\kappa\in(0,1)$. 
        \item\label{it.pbpunique} {[Uniform path-by-path uniqueness]} Any solution $Y$ to \eqref{eqn.ODEomega} starting at $x\in\Rd$ from time $s\ge0$
        is identical to $X^{s,x}_\cdot(\omega)$.
    \end{enumerate}
    \end{theorem}

    \begin{remark}\label{rem:improvement}
        \cref{thm.main} improves upon \cite[Theorem 1.2]{FF} by showing  joint continuity in all parameters and allowing the set of full measure on which the flow solves the equation to be uniform with respect to the initial data.
        Because $b$ is not continuous, showing that the continuous extension $X$ is a solution to \eqref{SDE} is a major obstacle.
        We overcome this thanks to the path-by-path regularizing estimates obtained by sewing techniques and Davie estimates, see \cref{lem.nonlinYoung,cor.intcont}. These arguments can be applied for other situations and it is expected that the conclusions of \cref{thm.main} hold for SDEs with bounded measurable drifts, more details are discussed in \cref{rmk.issue}.
    \end{remark}
    \begin{remark}\label{rem:flow}
        The semiflow $(X^{s,x}_t)$ in \cref{thm.main} is  constructed by taking the continuous extension of the family of strong solutions over the dyadics. Hence, it  inherits other properties and as such, it is adapted and is a flow of  homeomorphisms. In particular, it is possible to modify the event $\Omega_b$ in \cref{thm.main} so that for each $\omega\in\Omega_b$, the map $x\mapsto X^{s,x}_t(\omega)$ is a homeomorphism on $\Rd$. Because this property is not used in proving proving path-by-path uniqueness, we have omitted it in the statement of \cref{thm.main}. 
    \end{remark}

        \begin{remark}%
    In \cref{thm.main}, the dependence of $\Omega_b$ on $b$ seems necessary with our methods. The event $\Omega_{b}$ essentially contains $\omega$ such that 
    \begin{itemize}
        \item[(i)] the semiflow $X^{s,x}_t(\omega)$ is almost Lipschitz in $x$,
        \item[(ii)] the averaging operators
        \begin{align*}
            (t,x)\mapsto \int_0^t b(r,W_r(\omega)+x)\dd r,\,\int_0^t |b|(r,W_r(\omega)+x)\dd r
        \end{align*}
        are jointly locally H\"older continuous with exponent $(\alpha,1- \varepsilon)$ for some $\alpha\in(0,1)$ and some $\varepsilon>0$ sufficiently small.   
    \end{itemize}
    \end{remark}
    \begin{remark}
        In contrast to previously obtained strong uniqueness results showing uniqueness among a class of adapted solutions, \cref{thm.main} effectively drops these requirements via the $\omega$ by $\omega$ argument. In particular, we can allow for initial condition depending on future information of the driving process, e.g. $X_0(\omega)=W_1(\omega)$. Solutions of such equations are anticipating and fall outside the  scope of previous studies.  
    \end{remark}

    In the following corollary we observe a general principle: If over any finite period of time, an SDE generates a random continuous semiflow and path-by-path uniqueness holds, then the SDE is strongly complete.
    Although this fact has been observed in \cite{MR3567487} in the context of rough differential equations, singular SDEs form another distinct class of equations. It is therefore necessary to draw the connection.

    \begin{corollary}
    \label{cor.strcomplete}
      Suppose that for every $T>0$, \eqref{SDE} generates a random continuous semiflow over $[0,T]$ and that path-by-path uniqueness holds over $[0,T]$. Then \eqref{SDE} is strongly complete.
   \end{corollary}
  The proof of the above result are presented in 
\cref{sec:proof_main_results}.

\section{Davie estimates}\label{sec.Davie}
We show a variant of \emph{Davie estimates} in \cite[Propositions 2.1, 2.2]{Davie}. Herein, we fix a finite time horizon $T$. 
\begin{proposition}\label{prop.davie}
    Let $f$ be a Schwartz function in $\LL^q_p$ for some $(p,q)\in\cjj$.  Then there is a  constant $C=C(d,p,q,T)$ such that
    for every $(s,t) \in \simpl{0}{T}$ and $m\ge1$,%
    \begin{align}\label{gradf}
        \left\|\int_s^t\nabla f(r,W_r)\dd r\right\|_{L^m(\Omega)}\le C \Gamma\left(m(\frac12+\frac d{2p})+1\right)^\frac{1}{m} \|f\|_{\LL^q_p([s,t])}(t-s)^{\frac12-\frac1q-\frac d{2p}},
    \end{align}
    where $\Gamma(r)=\int_0^\infty u^{r-1}e^{-u}\dd u$ is the Gamma function.
    Consequently, for every bounded continuous function $f$ in $\LL^q_p([0,T])$, every $x,y\in\mR^d$, every $(s,t) \in \simpl{0}{T}$ and $m\ge1$, we have
    \begin{multline}\label{est:exp-diff}
        \left\|\int_s^t[f(r,W_r+x)-f(r,W_r+y)]\dd r\right\|_{L^m(\Omega)}
       \\\le C \Gamma\left(m(\frac12+\frac d{2p})+1\right)^\frac{1}{m} \|f\|_{\LL^q_p([s,t])}(t-s)^{\frac12-\frac1q-\frac d{2p}}|x-y|.
    \end{multline} 
\end{proposition}%
Contrary to Davie's original arguments (see \cite{F}),
the precise growth constant in \eqref{est:exp-diff} is not essential to our proof of \cref{thm.main}.  Nevertheless, it improves upon previous known estimates from \cite{rezakhanlou2014regular,F}. To prove this result, we follow a recent approach from \cite{Le2022} based on the quantitative John--Nirenberg inequality. 
We prepare two auxiliary lemmas, putting  
\begin{align*}
  p_t(x)\coloneqq\1_{(t>0)}(2\pi t)^{-\frac{d}{2}}e^{-\frac{|x|^2}{2t}}.  
\end{align*}

\begin{lemma}\label{lem.parRiesz}
    The operator $f\mapsto D(f)$ defined by
    \begin{align*}
        D(f)(t,x)
        \coloneqq\int_{\R\times\Rd}\nabla p_s(y)\cdot \nabla f(t-s,x-y)\dd (s,y),
    \end{align*}
   is  bounded  on $\LL^q_p$ for every $(p,q)\in(1,\infty)^2$.
\end{lemma}
\begin{proof}
    The operator $D$ is associated to the Fourier multiplier
    \begin{align*}
        \Phi(x_0,x)=c\frac{x^2}{x^2+\mathrm{i} x_0}, \quad x^2=x_1^2+...+x_d^2,
    \end{align*}
    for some absolute constant $c$, where $\mathrm{i}$ is the imaginary unit.
    A result from \cite[Corollary 1, pg. 234]{MR0262815} asserts that $D$ is a bounded operator on $\LL^q_p$ provided that $\Phi$ admits continuous (for $x_0,...,x_d\neq 0$) purely mixed derivatives of orders $k\le d+1$ such that
    \begin{align*}
        \left|x_{j_1}...x_{j_k}\partial^k_{x_{j_1}...x_{j_k}} \Phi \right|\le M 
    \end{align*}    
    for all distinct $j_1,...,j_k$ and for some finite constant $M$. Such condition can be verified directly. 
\end{proof}
\begin{lemma}[Quantitative John--Nirenberg inequality]\label{lem.vmo}
    Let $V$ be an adapted continuous process, $w$ be a deterministic control and $\alpha\in(0,1)$. Assume that 
    \begin{align*}
        \|\E_s|V_t-V_s|\|_{L^\infty(\Omega)}\le w(s,t)^\alpha \quad \forall (s,t)\in\simt.
    \end{align*}
    Then, there exists a finite constant $C_\alpha>0$ such that for every $(s,t)\in\simt$ and every $m\ge1$,
    \begin{align*}
        \|\sup_{u\in[s,t]}|V_u-V_s|\|_{L^m(\Omega)}\le C_\alpha \Gamma(m(1- \alpha)+1)^\frac{1}{m} w(s,t)^\alpha.
    \end{align*}
\end{lemma}
\begin{proof}
    Without the growth constant $\Gamma(m(1- \alpha)+1)^\frac{1}{m}$, %
    this result is a direct consequence of \cite[Exercise A.3.2]{MR2190038}. To obtain the precise growth constant, we follow \cite{Le2022} closely.
 Continuity and the assumption imply that
    \begin{align*}
         \| \E_\sigma|V_\tau-V_\sigma|\|_{L^\infty(\Omega)}\le w(s,t)^\alpha
    \end{align*}
    for every $(s,t)\in\simt$ and all stopping times $\sigma, \tau$ satisfying $s\le \sigma\le \tau\le t$. The class of all processes with such property is denoted by $\mathrm{VMO}^{(1/\alpha)-\var}$ in \cite[Section 3]{Le2022}. We then apply \cite[Corollary 3.5]{Le2022} to obtain the desired estimate.
\end{proof}
\begin{proof}[\textbf{Proof of \cref{prop.davie}}]
In view of \cref{lem.vmo}, it suffices to estimate the conditional second moment.
Define for each $(s,t)\in[0,T]_\leq^2$,
\begin{align}
    V_t\coloneqq\int_0^t\nabla f(r,W_r)\dd r,  \label{def:V}
    \tand I_{s,t}\coloneqq\frac12\mE_s\big|V_t-V_s\big|^2.
\end{align}
Using integration by parts, tower property of conditional expectation and Fubini's Theorem we have, almost surely,
\begin{align*} 
    I_{s,t}&=\mE_s\int_s^t\int^t_{r_2}\mE_{r_2}[\nabla f(r_1,W_{r_1})] \cdot \nabla f(r_2,W_{r_2}) \dd r_1\dd r_2
  \\&=-\mE_s\int_s^t\int^t_{r_2}\big[\int_{\mR^d}f(r_1,y+W_{r_2})\nabla  p_{r_1-r_2}(y)\dd y\big]\cdot  \nabla f(r_2,W_{r_2})\dd r_1\dd r_2 
  \\&=-\int_s^t\int^t_{r_2}\int_{\mR^d}\big[\int_{\mR^d}f(r_1,y+z+W_s)\nabla  p_{r_1-r_2}(y)\dd y\big] \cdot\nabla f(r_2,z+W_s) p_{r_2-s}(z)\dd z\dd r_1\dd r_2 
  \\&=:I_1+I_2,
\end{align*}
where
\begin{align*}
    I_1&=\int_s^t\int^t_{r_2}\int_{\mR^d}\big[\int_{\mR^d}f(r_1,y+z+W_s)\nabla  p_{r_1-r_2}(y)\dd y\big] \cdot f(r_2,z+W_s)\nabla  p_{r_2-s}(z)\dd z\dd r_1\dd r_2 ,
    \\I_2&=\int_s^t\int^t_{r_2}\int_{\mR^d}\big[\int_{\mR^d}\nabla f(r_1,y+z+W_s)\cdot \nabla p_{r_1-r_2}(y)\dd y\big]  f(r_2,z+W_s)p_{r_2-s}(z)\dd z\dd r_1\dd r_2.
\end{align*}

To estimate $I_1$, we apply H\"older's  inequality (below $p'$ satisfies $\frac{1}{p}+\frac{1}{p'}=1$) to see that for any $v\in\mR^d$, 
\begin{align*}
&\int_s^t\int^t_{r_2}\int_{\mR^d}\big[\int_{\mR^d}f(r_1,y+z+v)\nabla  p_{r_1-r_2}(y)\dd y\big]\cdot  f(r_2,z+v)\nabla  p_{r_2-s}(z)\dd z\dd r_1\dd r_2
	\\&\le \int_s^t\int^t_{r_2}\|f(r_1,z+v+\cdot)\|_{L^p(\mR^d)}\| \nabla  p_{r_1-r_2}\|_{L^{p'}(\mR^d)} \|f(r_2,v+\cdot)\|_{L^p(\mR^d)}\|\nabla  p_{r_2-s}\|_{L^{p'}(\mR^d)}\dd  r_1\dd r_2 
 \\&\lesssim\Vert f\Vert_{\LL_p^q([s,t])}\int_s^t\big(\int_{r_2}^t(r_1-r_2)^{(-\frac{1}{2}-\frac{d}{2p})\frac{q}{q-1}}\dd r_1\big)^{1-\frac{1}{q}}  \Vert f(r_2)\Vert_{L^{p}(\mR^d)}(r_2-s)^{-\frac{1}{2}-\frac{d}{2p}}\dd r_2
  \\&\lesssim\Vert f\Vert_{\LL_p^q([s,t])}\int_s^t (t-r_2)^{\frac{1}{2}-\frac{d}{2p}-\frac{1}{q}} \Vert f(r_2)\Vert_{L^{p}(\mR^d)}(r_2-s)^{-\frac{1}{2}-\frac{d}{2p}}\dd r_2 
 \\&\lesssim\Vert f\Vert_{\LL_p^q([s,t])}(t-s)^{\frac{1}{2}-\frac{d}{2p}-\frac{1}{q}}\int_s^t \Vert f(r_2)\Vert_{L^{p}(\mR^d)}(r_2-s)^{-\frac{1}{2}-\frac{d}{2p}}\dd r_2 
 \\&\lesssim\Vert f\Vert_{\LL_p^q([s,t])}(t-s)^{\frac{1}{2}-\frac{d}{2p}-\frac{1}{q}}\Vert f\Vert_{\LL_p^q([s,t])}(t-s)^{\frac{1}{2}-\frac{d}{2p}-\frac{1}{q}}
 \lesssim\Vert f\Vert_{\LL_p^q([s,t])}^2(t-s)^{1-\frac{2}{q}-\frac{d}{p}}.
\end{align*}
This yields that, almost surely,
\begin{align*}
   | I_1| 
   \lesssim\Vert f\Vert_{\LL_p^q([s,t])}^2(t-s)^{1-\frac{2}{q}-\frac{d}{p}}.
\end{align*}
To estimate $I_2$, we apply  H\"older inequality and \cref{lem.parRiesz} to see that
\begin{align*}
    |I_2|\lesssim \|f\|_{\LL^q_p([s,t])}\|\1_{[s,t]}f^{W_s}p^{s}\|_{\LL^{q'}_{p'}}
\end{align*}
where 
\begin{align*}
  f^{W_s}(r,z)=f(r,z+W_s),\quad p^s_r(z)=p_{r-s}(z),  
\end{align*}
and $p',q'$ denote the H\"older conjugates of $p,q$ respectively.
    Applying H\"older inequality again, we have
    \begin{align*}
        \|\1_{[s,t]}f^{W_s}p^{s}\|_{\LL^{q'}_{p'}}\le \|\1_{[s,t]}p^{s}\|_{\LL^{q/(q-2)}_{p/(p-2)}}\|f\|_{\LL^q_p([s,t])}.
    \end{align*}
    Using the fact that $(p,q)\in\cjj$ and some elementary calculations, we have
    \begin{align*}
       \|\1_{[s,t]}p^{s}\|_{\LL^{q/(q-2)}_{p/(p-2)}}=c(t-s)^{1-\frac dp-\frac2q} 
    \end{align*}
  for some constant $c$ depending on $p,q$. Combining with the previous estimates, we obtain that 
    \begin{align*}
        |I_{s,t}|\lesssim\Vert f\Vert_{\LL_p^q([s,t])}^2(t-s)^{1-\frac{2}{q}-\frac{d}{p}} \quad a.s.
    \end{align*}
Consequently, 
\begin{align*}
   \Vert\mE_s|V_t-V_s|\Vert_{L^\infty(\Omega)}\lesssim\Vert f\Vert_{\LL_p^q([s,t])}(t-s)^{\frac{1}{2}-\frac{1}{q}-\frac{d}{2p}}. 
\end{align*}
It is evident that $V$ is continuous. Applying \cref{lem.vmo}, we obtain \eqref{gradf}.

 Next we show \eqref{est:exp-diff}. By approximations, we can assume that $f$ has bounded continuous first derivatives. We observe that 
   \begin{align*}
       \int_s^tf(r,W_r+x)&-f(r,W_r+y)\dd r
      =  (x-y)\cdot\int_0^1\int_s^t\nabla f\big(r,W_r+\theta x+(1-\theta)y\big)\dd r\dd \theta.
   \end{align*}
  Then it follows from \eqref{gradf} that 
   \begin{align*}
       &\big\Vert\int_s^tf(r,W_r+x)-f(r,W_r+y)\dd r\big\Vert_{L^m(\Omega)}
       \\&\leq
\sup_{\theta\in[0,1]}\big\Vert\int_s^t\nabla f\big(r,W_r+\theta x+(1-\theta)y\big)\dd r\big\Vert_{L^m(\Omega)}|x-y|
       \\&\leq |x-y|\sup_{z\in\mR^d}\big\Vert\int_s^t\nabla f(r,W_r+z)\dd r\big\Vert_{L^m(\Omega)}
       \\&\les \Gamma\left(m(\frac12+\frac d{2p})+1\right)^\frac{1}{m} \|f\|_{\LL^q_p([s,t])}(t-s)^{\frac12-\frac1q-\frac d{2p}}|x-y|.
   \end{align*} %
   The proof is completed.
\end{proof}

\section{Path-by-path estimates} %
\label{sec:path_estimates}
    Let $f$ be a measurable function on $\R\times\Rd$. If $\omega$ is such that  $\int_0^T |f(r,W_r(\omega))|\dr<\infty$ then the function $t\mapsto \int_0^t f(r,W_r(\omega))\dr$ is continuous. The following result extends this argument for the function $(t,x)\mapsto\int_0^t f(r,W_r(\omega)+x)\dr$. The integrability condition  $\int_0^T |f(r,W_r(\omega)+x)|\dr<\infty$ is no longer sufficient and one has to replace it with the quantity $\Xi_{T,R}(f)(\omega)$ defined in \cref{lem.regularizing} below. We first prove a simple embedding.
    \begin{lemma}\label{lem.grr}
        Let $T>0$, $R>0$, $m>d$, $\alpha\in (\frac{1}{m},1 )$ and $\beta\in(\frac{d}{m} ,1)$.
        Then there exists a finite constant $C>0$ independent from $T,R$ such that for any continuous function $g:[0,T]\times B_R\to \R$, one has
            \begin{multline}\label{est.grr}
                \sup_{(s,t,x,y)\in[0,T]^2\times B_R^2}\frac{|\delta g_{s,t}(x)-\delta g_{s,t}(y)|^m }{|t-s|^{m\alpha-1}|x-y|^{m\beta-d}}
                \\\le C\iiiint_{[0,T]^2\times B_R^2} \frac{|\delta g_{\bar s,\bar t}(\bar x)-\delta g_{\bar s,\bar t}(\bar y)|^m }{|\bar t-\bar s|^{m\alpha+1}|\bar x-\bar y|^{m\beta+d}}  \dd\bar s \dd\bar  t  \dd\bar  x \dd\bar  y.
            \end{multline}
    \end{lemma}
    
    \begin{proof}
        Note that the inequality is invariant under scaling, hence we can assume that $T=R=1$. 
        First apply Sobolev embedding in $\Rd$, we see that the left-hand side of \eqref{est.grr} is at most a constant multiple of 
        \begin{align*}
            \iint_{ B_1^2} \frac{|\delta g_{ s, t}(\bar x)-\delta g_{ s, t}(\bar y)|^m }{| t- s|^{m\alpha+1}|\bar x-\bar y|^{m\beta+d}}  \dd\bar  x \dd\bar  y.
        \end{align*}
        We note that $\delta g_{ s, t}(\bar x)-\delta g_{ s, t}(\bar y)= \delta h_{s,t}$, where $h_t= g_{ t}(\bar x)- g_{ t}(\bar y)$. Thus, we can apply the Sobolev embedding in $\R$ to the function $h$ to see that the integral above is bounded by the right-hand side of \eqref{est.grr}.
    \end{proof}
    \begin{lemma}\label{lem.regularizing}
        Let $R>0$, $\alpha\in(0,\frac{1}{2}-\frac{d}{2p}-\frac{1}{q})$ for some $(p,q)\in\cjj$ and $\varepsilon \in (0,1)$. There exist a function $\Xi_{T,R,\alpha,\varepsilon}=\Xi_{T,R}:\LL^q_p\to L^1(\Omega)$ and a finite constant $C=C(\alpha,\varepsilon)$ such that
        \begin{enumerate}[(i)]
            \item \label{(i)}$\Xi_{T,R}$ satisfies the triangle inequality, i.e. for every $f,g\in \LL^q_p$, one has that  $\Xi_{T,R}(f+g)\le \Xi_{T,R}(f)+\Xi_{T,R}(g)$ a.s.;
             \item\label{it.xif} for every $f\in\LL^q_p$,
            \begin{align}
                \E[\Xi_{T,R}(f)]\les \|f\|_{\LL^q_p([0,T])};
                \label{est.Xif}
            \end{align}
            \item \label{(ii)}for every bounded measurable function $f$ in $\LL^q_p$, there exists an event $\Omega^\prime_{f,T,R}$ of full measure such that for every $\omega\in \Omega^\prime_{f,T,R}$, $\Xi_{T,R}(f)(\omega)$ is finite; for every  $(s,t)\in\simt$ and every $x,y\in B_R$, 
            \begin{align}\label{est.asLip}
                \Big|\int_s^t [f(r,W_r(\omega)+x)-f(r,W_r(\omega)+y)]\dd r\Big|&\le C\Xi_{T,R}(f)(\omega)|x-y|^{1- \varepsilon}(t-s)^{\alpha},
                \\\Big|\int_s^t f(r,W_r(\omega))\dd r\Big|&\le C\Xi_{T,R}(f)(\omega)(t-s)^\alpha.
                \label{est.asbounded}
            \end{align}
        \end{enumerate}
     \end{lemma}
    \begin{proof}
        We choose and fix $m\ge1$ such that 
        \begin{align}\label{cond.m1}
            \frac{1}{m}< \frac{1}{2}-\frac{d}{2p}-\frac{1}{q}-\alpha \tand \frac{d}{m}<\varepsilon.
        \end{align} 
        For each function $f\in\LL^q_p$, we define
        \begin{multline}\label{def.Xi}
            \Xi_{T,R}(f)=\Big(   \iint_{[0,T]^2}\frac{\Big|\int_s^tf(r,W_r)\dd r\Big|^m}{|t-s|^{\alpha m+2}}\dd s\dd t
            \\+\iiiint_{[0,T]^2\times B_R^2}\frac{\Big|\int_s^t[f(r,W_r+x)-f(r,W_r+y)]\dd r\Big|^m}{|t-s|^{\alpha m+2}|x-y|^{(1-\varepsilon)m+2d}}\dd s\dd t\dd x\dd y \Big)^{1/m}.
        \end{multline}
        It is obvious that \ref{(i)} holds.
        Using \eqref{est:exp-diff} and the following estimate from \cite[Lemma 4.5]{LL} 
         \begin{align}\label{tmp.kryl}
             \mE\Big(\int_s^t f(r,W_r)\dd r\Big)^m\les  \Vert f\Vert_{\LL_p^q([s,t])}^m(t-s)^{m(1-\frac{d}{p}-\frac{2}{q})},
         \end{align} 
        we see that \ref{it.xif} holds.  
        It remains to show that \ref{(ii)} holds.
        Assume first that $f$ is bounded continuous so that the map $(t,x)\mapsto\int_0^t f(r,W_r(\omega)+x)\dr$ is continuous for each $\omega$ inside an event $\Omega^\prime$ of full measure (for instance, choosing $\Omega^\prime$ so that $t\mapsto W_t(\omega)$ is measurable for $\omega \in \Omega^\prime$). 
        Define $\Omega^\prime_{f,T,R}\coloneqq\Omega^\prime\cap \{\omega:\Xi_{T,R}(f)(\omega)<\infty\}$ which is an event of full measure.
        From \eqref{est:exp-diff} and \eqref{tmp.kryl}, applying \cref{lem.grr} and Sobolev embedding in $\R$ (with $\gamma_1=\alpha+\frac{1}{m} $ and $\gamma_2=1-\varepsilon+\frac{d}{m}$), 
        there exists a constant $C$  such that \eqref{est.asLip} and \eqref{est.asbounded} hold.

        Below, we remove the continuity assumption  on $f$.

        \textit{Step 1.} We show \ref{(ii)} for $f=\1_U$ where $U$ is an open set of finite measure in $\R\times\Rd$. By Urysohn lemma, there exists a sequence of increasing continuous functions $f^n$ converging pointwise to $f$. From \ref{(i)} and \ref{it.xif}, we can choose a further subsequence, still denoted by $(f^n)$ such that $\lim_n\Xi_{T,R}(f^n)=\Xi_{T,R}(f)$ a.s., w.l.o.g. on $\Omega^\prime_{f,T,R}$.
        We then apply \eqref{est.asLip} and \eqref{est.asbounded} for $f^n$ and take limit in $n$ to obtain \ref{(ii)}.

        \textit{Step 2.} We show \ref{(ii)} for a general bounded measurable $f$ in $\LL^q_p$. Let $M$ be a constant such that $|f|\le M$.
        By Lusin theorem, for any $n \in \mathbb{N}$, there exists a continuous function $f^{n}$ and an open set $U^{n}$ whose Lebesgue measure is not more than $2^{-n}$ such that $|f-f^{n}|\le 2M \1_{U^{n}}$.
        This implies that $\lim_{n}f^{n}=f$ in $\LL^q_p$. Using \ref{it.xif}, we can choose subsequences, still denoted by $n$ such that a.s., again w.l.o.g. on $\Omega^\prime_{f,T,R}$,
        \begin{align}\label{tmp.limitxi}
            \lim_n\Xi_{T,R}(\1_{U^{n}})=0
            \tand \lim_n \Xi_{T,R}(f^{n})=\Xi_{T,R}(f)
            .
        \end{align}
        Let  $\omega\in \Omega^\prime_{f,T,R}$ and define 
        \begin{align*}
            A[f](\omega) =\big|\int_s^t [f(r,W_r(\omega)+x)-f(r,W_r(\omega)+y)]\dd r\big|.
        \end{align*}
        Then
        \begin{align}\label{tmp.ffn}
            A[f](\omega)\le A[f^{n}](\omega)+A[f-f^{n}](\omega).
        \end{align}
        Applying \eqref{est.asLip} for the continuous function $f^{n}$, we have 
        \begin{align*}
            A[f^{n}](\omega)\le \Xi_{T,R}(f^{n})(\omega)|x-y|^{1- \varepsilon}(t-s)^\alpha \quad \forall (s,t,x,y)\in\simt\times B_R^2.
        \end{align*}     
        Applying \eqref{est.asLip} and \eqref{est.asbounded} for $\1_{U^{n}}$, we have 
        \begin{align*}
            A[f-f^{n}](\omega)
            &\le 2M\int_s^t[\1_{U^{n}}(r,W_r(\omega)+x)+\1_{U^{n}}(r,W_r(\omega)+y)]\dr
            \\&\le 4M \Xi_{T,R}(\omega)(\1_{U^{n}})(t-s)^\alpha \quad  \forall (s,t,x,y)\in\simt\times B_R^2.
        \end{align*}
        Hence, from \eqref{tmp.ffn}, taking limit in $n$ and using \eqref{tmp.limitxi}, we see that \eqref{est.asLip} holds. The estimate \eqref{est.asbounded} is obtained in a similar way.          
     \end{proof}
     \begin{remark}
         Via a truncation procedure as is done in the last step in the proof of \cref{lem.nonlinYoung} below, one can remove the assumption of boundedness in \cref{lem.regularizing}\ref{(ii)}. The current formulation is sufficient for our purpose.
     \end{remark}
    The following result provides an alternative perspective to  \cite[Lemmas 3.3 and 3.4]{Davie} and \cite[Lemmas 3.3 and 3.4]{AVS}.
    \begin{lemma}\label{lem.nonlinYoung}
        Let  $f:[0,T]\times\Rd \to \mathbb{R}$ be a function in $\LL^q_p$ for some $(p,q)\in\cjj$, $R>0$ and $\alpha \in (0,\frac{1}{2}-\frac{d}{2p}-\frac{1}{q})$. Then there exists an event $\Omega^\prime_{f,T,R}$ with full probability such that for every $\omega\in \Omega^\prime_{f,T,R}$,
        every  $\varepsilon\in(0,\alpha)$, there is a deterministic constant $C=C(\varepsilon,\alpha)$ such that for every $(s,t) \in \simpl{0}{T}$, every function  $\psi:[0,T]\to B_R$  of finite variation,
        we have $\int_{0}^{T}|f(r,W_r(\omega)+\psi_r)|\dd r<\infty$, 
        \begin{gather}\label{est.young}
            \Big|\int_s^t f(r,W_r(\omega)+\psi_r)\dd r-\int_s^t f(r,W_r(\omega)+\psi_s)\dd r\Big|\le C \Xi_{T,R}(f)(\omega)[\psi] _{\var;[s,t]}^{1- \varepsilon}(t-s)^{\alpha}
            \\\shortintertext{and}
            \Big|\int_s^t f(r,W_r(\omega)+\psi_r)\dd r\Big|\le C \Xi_{T,R}(f)(\omega)(1+|\psi_s|^{1-\varepsilon}+[\psi] _{\var;[s,t]}^{1- \varepsilon})(t-s)^{\alpha}.\label{est.intfome}
        \end{gather}
    \end{lemma}
    \begin{proof}
        Let $\Omega^\prime_{f,T,R}$ be as in \cref{lem.regularizing} and fix $\omega\in \Omega^\prime_{f,T,R}\cap \Omega^\prime_{|f|,T,R}$.

        \textit{Step 1.} We first show \eqref{est.young}  assuming that $f$ is a bounded continuous function.
        For $(s,t) \in \simpl{0}{T}$ define $A_{s,t}=\int_s^tf(r,W_r(\omega)+\psi_s)\dd r$ so that, for $ u \in [s,t]$,
        \begin{align*}
            \delta A_{s,u,t}\coloneqq A_{s,t}-A_{s,u}-A_{u,t} =\int_u^t[f(r,W_r(\omega)+\psi_s)-f(r,W_r(\omega)+\psi_u)]\dd r.
        \end{align*}
        Applying \eqref{est.asLip}, we have
        \begin{align*}
            |\delta A_{s,u,t}|\les \Xi_{T,R}(f)(\omega) |\psi_u-\psi_s|^{1- \varepsilon}(t-s)^{\alpha}
            \les \Xi_{T,R}(f)(\omega) [\psi]_{\var;[s,t]}^{1- \varepsilon} (t-s)^{\alpha}.
        \end{align*}
        That $\varepsilon<\alpha$ ensures that $1- \varepsilon+\alpha>1$.
        Since $f$ is continuous, we have for every $t\in[0,T]$,
        \begin{align*}
            \int_0^t f(r,W_r(\omega)+\psi_r)\dd r=\lim_{|\pi|\downarrow0}\sum_{[a,b]\in \pi}A_{a,b},
        \end{align*}
        where $\pi$ is any partition of $[0,t]$.
        We apply the sewing lemma \cite[Theorem 2.2 and Remark 2.3]{FH} formulated with controls to obtain \eqref{est.young}. 
        
        \textit{Step 2.} We show the result for any bounded measurable function $f$,  using similar arguments as in \cref{lem.regularizing}. 

        First, using Urysohn lemma, one can show that \eqref{est.young} holds for $f=\1_U$ for any open set $U$ of finite measure. 
        In particular, we have 
        \begin{align}\label{tmp.1u}
            \Big|\int_s^t \1_U(r,W_r(\omega)+ \psi_r)\dr\Big|\les \Xi_{T,R}(\1_U)(\omega)(1+R^{1-\varepsilon}+[\psi]_{\var;[s,t]}^{1- \varepsilon})(t-s)^\alpha 
        \end{align}
        for every $(s,t)\in\simt$ and every $\psi:[0,T]\to B_R$  of finite variation.
        
        Next, applying Lusin theorem, we can find for each integer $n\ge1$ a continuous bounded function
        $f^{n}$ and an open set $U^{n}$ whose Lebesgue measure is not more than $2^{-n}$ such that $|f-f^{n}|\le 2M \1_{U^{n}}$, where $M$ is a constant such that $|f|\le M$.
        This implies that $\lim_{n}f^{n}=f$ in $\LL^q_p$. For each $s\le t$, we put
        \begin{align*}
            J[f] =\int_s^t f(r,W_r(\omega)+\psi_r)\dd r-\int_s^t f(r,W_r(\omega)+\psi_s)\dd r
        \end{align*}
        so that $J[f]=J[f^{n}]+J[f-f^{n}]$.
        By the previous step, 
        \begin{align*}
            |J[f^{n}]|\le C \Xi_{T,R}(f^{n})(\omega)[\psi]_{\var;[s,t]}^{1- \varepsilon}(t-s)^\alpha.
        \end{align*}
        Using  \eqref{tmp.1u}, \eqref{est.asLip} and  \eqref{est.asbounded}, we have
        \begin{align*}
            |J[f-f^{n}]|
            &\les M \int_s^t \1_{U^n}(r,W_r(\omega)+ \psi_r)\dr+M\sup_{x\in B_R}\int_s^t\1_{U^n}(r,W_r(\omega)+x)\dr
            \\&\les M T^\alpha(1+R^{1-\varepsilon}+[\psi]^{1-\varepsilon}_{\var;[s,t]}) \Xi_{T,R}(\1_{U^n})(\omega).
        \end{align*}
        By \cref{lem.regularizing}\ref{(i)}\ref{it.xif}, we can further choose a subsequence, still denoted by $n$, such that, w.l.o.g. on $\Omega^\prime_{f,T,R}$,
        \begin{align*}
             \lim_n \Xi_{T,R}(f^n)=\Xi_{T,R}(f)\tand\lim_n \Xi_{T,R}(\1_{U^n})=0.
        \end{align*}
        We emphasize that the null events only depend on $f,T,R$ and are independent from $s,t, \psi$. Then, by passing through the limit in $n$, we obtain that
        \begin{align*}
            |J[f]|\le  C \Xi_{T,R}(f)(\omega)[\psi]_{\var;[s,t]}^{1- \varepsilon} (t-s)^\alpha,
        \end{align*}
        which shows \eqref{est.young}. 
        Additionally, from \eqref{est.asLip}, we have
        \begin{align*}
            \Big|\int_s^t f(r,W_r(\omega)+\psi_s)\dd r-\int_s^t f(r,W_r(\omega))\dd r\Big|\le C \Xi_{T,R}(f)(\omega)|\psi_s|^{1- \varepsilon}(t-s)^{\alpha}.
        \end{align*}
        Combining with \eqref{est.young}, we obtain \eqref{est.intfome}.

        \textit{Step 3.} Consider now the case $f \in\LL^q_p$. 
        For each integer $M\ge1$, define $f^M=f\1_{(|f|\le M)}$. Note that $(f^M)$ (respectively $(|f^M|)$) is  a sequence of bounded functions converging to $f$ (respectively $|f|$) in $\LL^q_p$. We can choose a sequence $(M_n)$ such that $(\Xi_{T,R}(f^{M_n}),\Xi_{T,R}(|f^{M_n}|))$ converges to $(\Xi_{T,R}(f),\Xi_{T,R}(|f|))$ on $\Omega^\prime_{f,T,R}$, again w.l.o.g. 
        Define $G=\{(r,z)\in[0,T]\times\Rd: \lim_nf^{M_n}(r,z)=f(r,z)\text{ and }\lim_n|f^{M_n}|(r,z)=|f|(r,z)\}$ and note that $|G^c|=0$. By \eqref{est.Xif}, we can assume w.l.o.g. that $\Xi_{T,R}(\1_{G^c})(\omega)=0$ for every $\omega\in\Omega_{f,T,R}$. By \eqref{est.intfome}, for  $\omega\in\Omega_{f,T,R}$. we have
  \begin{align*}
    \int_0^T \1_{G^c}(r,W_r(\omega)+\psi_r)dr=0 .
  \end{align*}
  This implies that $f^{M_n}(r,W_r(\omega)+\psi_r)\to f(r,W_r(\omega)+\psi_r)$ and $|f^{M_n}|(r,W_r(\omega)+\psi_r)\to |f|(r,W_r(\omega)+\psi_r)$ for a.e. $r\in [0,T]$. 
  Using this, monotone convergence theorem and \eqref{est.intfome}, we then have
  \begin{align*}
    \int_0^T |f(r,W_r(\omega)+\psi_r)|dr
    &=\lim_n\int_0^T |f^{M_n}(r,W_r(\omega)+\psi_r)|dr
    \\&\quad\le \Xi_{T,R}(|f|)(\omega)(1+R^{1-\varepsilon}+[\psi]^{1- \varepsilon}_{\var;[0,T]})T^\alpha.
  \end{align*}
  This shows that  for every $\psi:[0,T]\to B_R$  of finite variation, the function $t\mapsto f(t,W_t(\omega)+\psi_r)$ is integrable on $[0,T]$. Furthermore, by dominated convergence theorem, we have for every $t\in [0,T]$, 
  \begin{align*}
    \int_0^t f(r,W_r(\omega)+\psi_r)dr
    &=\lim_n\int_0^t f^{M_n}(r,W_r(\omega)+\psi_r)dr.
  \end{align*}
    An analogous argument shows that for every $(t,x)\in[0,T]\times B_R$,
    \begin{align*}
        \int_0^t f(r,W_r(\omega)+x)dr
        &=\lim_n\int_0^t f^{M_n}(r,W_r(\omega)+x)dr.
      \end{align*}
    Noting that $f^{M_n}$ is bounded measurable, hence, the previous step shows that \eqref{est.young} and \eqref{est.asbounded} hold with $f^{M_n}$ in place of $f$. From here, it suffices to send $n\to\infty$ and apply the above limit identities to obtain \eqref{est.young} and \eqref{est.asbounded} for $f$.
    This concludes the proof.
\end{proof}

    \begin{corollary}\label{cor.intcont}
        Let $f$ be a function in $\LL^q_p$ for some $(p,q)\in\cjj$.
        Let $(\abz, \rho)$ be a metric space and $(t,z)\mapsto \psi^z_t$ be a bounded function mapping from  $[0,T]\times\abz$ to $\mathbb{R}$  such that
        \begin{enumerate}[(i)]
            \item\label{it.var} for each $z\in\abz$, $t\mapsto \psi^z_t$ has finite variation and $\sup_{z\in\abz}(|\psi^z_0|+[\psi^z]_{\var;[0,T]})<\infty$;
            \item\label{it.equicont} 
            for each $t \in [0,T]$, $z\mapsto\psi^{z}_t$ is continuous.
        \end{enumerate}
        Then there exist an event $\Omega^\prime_{f,T}$ of full measure and a sequence of bounded continuous functions $(f^n)$ which are independent from $\psi$ and $(\abz,\rho)$, such that $\lim_nf^n=f$ in $\LL^q_p$ and for every $\omega\in \Omega^\prime_{f,T}$,
        \begin{align*}
        \lim_n\int_0^tf^n(r,W_r(\omega)+\psi^z_r)\dr=\int_0^t f(r,W_r(\omega)+\psi^z_r)\dr \text{ uniformly in }(t,z)\in[0,T]\times\abz. 
        \end{align*}
        Consequently, the map
        \begin{align*}
   (t,z)\mapsto\int_0^tf(r,W_r(\omega)+\psi^z_r)\dr
        \end{align*}
        is continuous on $[0,T]\times\abz$.
    \end{corollary}
    \begin{proof}
        For each integer $i\ge1$, let $\Omega^\prime_{f,T,i}$ be as in \cref{lem.nonlinYoung}. In view of \eqref{est.Xif}, we can choose a sequence  of bounded continuous functions $(f^n)$ such that $\lim_n \Xi_{T,i}(f^n-f)=0$, w.l.o.g. on $\Omega^\prime_{f,T,i}$ for all integer $i\ge1$. 
        Because $\psi$ is bounded, there exists $j \in \mathbb{N}$ such that $\psi^z_t \in B_j$ for all $(t,z)\in[0,T]\times\abz$. Applying \eqref{est.intfome} and \ref{it.var}, we have for any $\omega \in \Omega^\prime_{f,T}\coloneqq\bigcap_{i \in \mathbb{N}} \Omega^\prime_{f,T,i}$, 
        \begin{align*}
            \lim_n\int_0^t f^n(r,W_r(\omega)+ \psi^z_r)\dr=\int_0^t f(r,W_r(\omega)+ \psi^z_r)\dr 
            \text{ uniformly in  }(t,z)\in[0,T]\times \abz.
        \end{align*}
        It remains to show that the function $(t,z)\mapsto \int_0^t f^n(r,W_r(\omega)+\psi^z_r)\dr$ is continuous for each $n$. To see that, fix $(t,z)$ and take a sequence $(t^i,z^i)_i$ converging to $ (t,z)$. Then
      \begin{align*}
          &\lim_i \Big|\int_0^{t^i} f^n(r,W_r(\omega)+ \psi^{z^i}_r) dr-\int_0^t f^n(r,W_r(\omega)+ \psi^{z}_r) dr\Big|\\
          \leqslant\lim_i&\Big|\int_0^{t^i} f^n(r,W_r(\omega)+ \psi^{z^i}_r) dr - \int_0^{t} f^n(r,W_r(\omega)+ \psi^{z^i}_r) dr\Big| \\
          &+ \lim_i\Big|\int_0^{t} f^n(r,W_r(\omega)+ \psi^{z^i}_r) dr - \int_0^t f^n(r,W_r(\omega)+ \psi^{z}_r) dr\Big|\\
          &\leqslant \lim_i \|f^n\|_\infty |t-t^i| + \lim_i\int_0^{T} |f^n(r,W_r(\omega)+ \psi^{z^i}_r) dr -f^n(r,W_r(\omega)+ \psi^{z}_r)|dr.
      \end{align*}
      The convergence to $0$ for the first term in the above is clear. For the second one, we can apply dominated convergence due to \eqref{est.intfome}; then convergence to $0$ for the second term holds by continuity of $f^n$ and Assumption \ref{it.equicont}.

    \end{proof}
     \begin{remark}\label{rmk:others} We record some consequential observations which may be useful for other purposes. 

        (i) From \cref{lem.nonlinYoung}, it follows that for a.s. $\omega$, a solution of \eqref{eqn.ODEomega} is also a solution in the framework of nonlinear Young integrals. More precisely, for a.s. $\omega$, if $Y$ is a solution to \eqref{eqn.ODEomega}, then there is a control $w$ and a number $\beta>1$ such that for $\psi=Y-W$,
        \begin{align*}
            |\psi_t-\psi_s-\int_s^tb(r,W_r(\omega)+\psi_s)\dr|\le w(s,t)^{\beta} \quad \forall (s,t)\in\simt.
        \end{align*}
        Via the sewing lemma, this means that $\psi$ is a solution to the nonlinear Young integral equation (\cite{MR3581224})
        \begin{align*}
            \psi_t=\psi_s+\int_s^t b^{W(\omega)} (\dr, \psi_r) \quad \forall (s,t)\in\simt,
        \end{align*}
        where $\int_s^t b^{W(\omega)}(\dr, \psi_r) $ is the nonlinear Young integral defined as the limit of the Riemann sums
        \begin{align*}
            \sum_{[u,v]}\int_u^v b(r, W_r(\omega)+ \psi_u) \dr.
        \end{align*}
        While nonlinear Young integral equations have been a central theme in previous works \cite{anzeletti2021regularisation,CatellierGubinelli,GaleatiGubinelli,HarangPerkowski,GaleatiGerencser}, it was not known if the ODE \eqref{eqn.ODEomega} under the Krylov-R\"ockner condition \eqref{LpLq} can be formulated in this framework.

        (ii) From \cref{cor.intcont}, it follows that for a.s. $\omega $, a solution of \eqref{eqn.ODEomega} is also a \emph{regularized solution} in the following sense: For a.s. $\omega$, if $Y$ is a solution to \eqref{eqn.ODEomega} then there exists a sequence of bounded continuous function $(b^n)$ and a continuous function $V:[0,T]\to\Rd$ such that
        $V_t=\lim_n\int_0^t b^n(r,Y_r)\dr$ uniformly on $[0,T]$
        and $Y_t=Y_0+V_t+W_t(\omega)$ for all $t\in[0,T]$.
        Regularized solutions of stochastic differential equations appear in \cite{anzeletti2021regularisation,BC,athreya2020well,butkovsky2023stochastic}.
    \end{remark} 
\section{Proof of main results} %
\label{sec:proof_main_results}
    We outline the main steps.
    For a given $\omega$ and a solution $Y$ to \eqref{eqn.ODEomega}, we obtain in \cref{prop.apri1} a priori estimates on the variations of $Y-W$ on arbitrary intervals. This allows us to show the existence of random H\"older continuous semiflow $(X^{s,x}_t)_{s,t,x}$ over arbitrary finite time period $[0,T]$, see \cref{lem.semiflow}. 
    We then show in \cref{prop.regularityabs} another set of a priori estimates on $Y_t-X^{s,Y_s}_t(\omega)$ for any given $(s,t)\in\simt$. Having these properties at our disposal, we proceed to prove path-by-path uniqueness, i.e. \cref{thm.main}\ref{it.pbpunique}. By utilizing path-by-path uniqueness, we can extend the random semiflow constructed in \cref{lem.semiflow} over finite time periods to all nonnegative times, thereby showing \cref{thm.main}\ref{it.complete}. 

    \begin{proposition}\label{prop.apri1}
    Let 
    $R>0$, $(u,v)\in\simt$, $\varepsilon\in(0,1)$ and $\alpha \in (0,\frac{1}{2}-\frac{d}{2p}-\frac{1}{q})$. Let $\omega\in \Omega^\prime_{|b|,T,R}$ where $ \Omega^\prime_{|b|,T,R}$ is the event in \cref{lem.nonlinYoung}. In particular, \eqref{est.young} and \eqref{est.intfome} hold with $f=|b|$ and any function $\psi$ of finite variation taking values in $B_R$.
	Let $Y:[u,v]\to\Rd$ be a solution to \eqref{eqn.ODEomega} on $[u,v]$ for such $\omega$, i.e. $\int_u^v|b(r,Y_r)|\dr<\infty$ and
    \begin{align*}
        Y_t=Y_s+\int_s^t b(r,Y_r)\dr+W_t(\omega)-W_s(\omega) \quad \forall (s,t)\in \simpl{u}{v}.
    \end{align*}
    Assume that  $Y_t- W_t(\omega)\in B_R$ for all $t\in[u,v]$. 
	Then there exists a finite constant $C=C(\varepsilon, \alpha)$ such that
        \begin{gather}
            \label{est.apribY}
            \int_u^v|b|(r,Y_r)\dr\le 
            C\Xi_{T,R}(|b|)(\omega)(1+R^{1-\varepsilon}) (v-u)^\alpha
            +
            C\big(\Xi_{T,R}(|b|)(\omega)(v-u)^\alpha\big)^{\frac1 \varepsilon}.
        \end{gather}    
    \end{proposition}
    \begin{proof}
        Since $\omega\in \Omega^\prime_{|b|,T,R}$ is fixed, we omit the dependence on $\omega$ in this proof.
        Define $\psi=Y-W$ and 
        note that
        \begin{align*}
            [\psi]_{\var;[u,v]}\le\int_u^v|b|(r,Y_r)\dr.
        \end{align*}
        Applying \eqref{est.intfome}, we have
        \begin{align*}
            \int_u^v|b|(r,Y_r)\dr&\le 
             C\Xi_{T,R}(|b|)(1+R^{1-\varepsilon}) (v-u)^\alpha
            \\&\quad+C \Xi_{T,R}(|b|)\left(\int_u^v|b|(r,Y_r)\dr\right)^{1- \varepsilon}(v-u)^\alpha.
        \end{align*}
        Applying Young's inequality, we have for every $\varepsilon'>0$
        \begin{align*}
            \Xi_{T,R}(|b|)& \Big|\int_u^v|b|(r,Y_r)\dd r\Big|^{1- \varepsilon}(v-u)^\alpha
            \le \varepsilon' \int_u^v|b|(r,Y_r)\dd r+C_{\varepsilon'}\big(\Xi_{T,R}(|b|)(v-u)^\alpha\big)^{\frac1 \varepsilon}.
        \end{align*}
        We choose $\varepsilon'$ sufficiently small to get \eqref{est.apribY}.
    \end{proof}
    \begin{proposition}[Random H\"older semiflow]\label{lem.semiflow} Let $T>0$.  Under
    \eqref{LpLq}, there exists an almost surely continuous semiflow to
    \eqref{SDE} over $[0,T]$  which is locally $\kappa$-H\"older continuous in space for every
    $\kappa\in(0,1)$. In other words, there exists a jointly measurable map $(s,
    t,x,\omega)\mapsto X^{s,x}_t(\omega)$ defined on $\simt\times\Rd\times
    \Omega$ and a set $\Omega_{b,T}^{flow}$ of full measure such that for $\omega\in \Omega_{b,T}^{flow}$, $(s,t,x)\mapsto X^{s,x}_t(\omega)$ is a
    continuous semiflow generated by \eqref{eqn.ODEomega} which is locally
    $\kappa$-H\"older continuous in space for every $\kappa\in(0,1)$.
    \end{proposition}
    \begin{proof}
        For each $s,x$, let $(\tilde X^{s,x}_t)_{t\in[s,T]}$ be the unique strong solution to \eqref{SDE} started from $x$ at time $s$. By definition of a solution, for every $\omega$ in a set of full measure that depends on $s$ and $x$, we have $\int_s^T|b(r,\tilde X^{s,x}_r(\omega))|\dd r<\infty$ and
        \begin{align}\label{tmp.eqn-initial}
            \tilde X^{s,x}_t(\omega)=x+\int_s^tb(r,\tilde X^{s,x}_r(\omega))\dd r+W_t(\omega)-W_s(\omega) \quad \forall t\in[s,T].
        \end{align}

        Let $m\ge2$ be a fixed number.
        By \cite[Theorem 1.2 (1.2)]{GL}
        \begin{align}\label{est:stab}
            \|\tilde X^{s,x}_t-\tilde X^{s,y}_t\|_{L^m(\Omega)}\lesssim \Vert x-y\Vert_{L^m(\Omega)}
        \end{align} 
        for any initial random points $x,y$ which are $\cff_s$-measurable.
        Applying \cite[Lemma 4.5]{LL} together with Girsanov Theorem (same proof of \cite[Corollary 3.4]{Kr-Ro} with keeping track of the time regularity) %
        \begin{align*}
            \|\tilde X^{s,x}_t-\tilde X^{s,x}_{t'}\|_{L^m(\Omega)}\lesssim\big\Vert\int_{t'}^tb(r,\tilde{X}_r^{s,x})\dd r\big\Vert_{L^m(\Omega)}+\|W_{t}-W_{t'}\|_{L^m(\Omega)}\lesssim |t-t'|^{\frac{1}{2} }.
        \end{align*}
        Using \eqref{est:stab} and pathwise uniqueness, we have for $s'<s$, 
        \begin{align*}
            \|\tilde X^{s,x}_t-\tilde X^{s',x}_t\|_{L^m(\Omega)}
            =\|\tilde X^{s,x}_t-\tilde X^{s,\tilde X^{s',x}_s}_t\|_{L^m(\Omega)}
            \lesssim \|x-\tilde X^{s',x}_s\|_{L^m(\Omega)}
            \lesssim |s-s'|^{\frac12}.
        \end{align*}
        It follows, that, for $x,y$ deterministic,
        \begin{align} \label{eq:cont}
            \|\tilde X^{s,x}_t-\tilde X^{s',y}_{t'}\|_{L^m(\Omega)}\lesssim |x-y|+ |t-t'|^{\frac{1}{2} }+|s-s'|^{\frac12}.
        \end{align}
        Since $m$ can be arbitrarily large, applying  Kolmogorov continuity criterion (\cite[Theorem (2.1) Chapter I]{MR1725357}), we see that for a.s. $\omega$, the map $(s,t,x)\mapsto\tilde X^{s,x}_t(\omega)$ 
        is locally H\"older continuous on 
        \[
            \{(s,t,x)\in\simt\times\Rd: s,t,x \text{ are dyadic}\}
        \] 
        with exponents $(\alpha',\beta,\kappa)\in(0,1/2)\times(0,1/2)\times(0,1)$. Because for each $s,x$, $(\tilde X^{s,x}_t(\omega))_{t\in[s,T]}$
        is continuous, the map $(s,t,x)\mapsto\tilde X^{s,x}_t(\omega)$ is locally H\"older continuous on 
        \[
            G\coloneqq\{(s,t,x)\in\simt\times\Rd: s,x \text{ are dyadic}\}
        \]
        with the same exponents.

        Let $\hat{\Omega}_{b,T}$ be the event of full measure on which \eqref{tmp.eqn-initial} holds whenever $s,x$ are dyadic, and $\tilde X$ is locally H\"older continuous on $G$.
        For each $\omega\in \hat{\Omega}_{b,T}$, let $X(\omega)$ be the (unique) continuous extension of $\tilde X(\omega)\big|_G$ to $\simt\times\Rd$. We show that $(X^{s,x}_t)_{s,t,x}$ is the desired semiflow. H\"older regularity is clear and hence, we focus on showing the semiflow properties. 
        Since $b$ is not continuous, it is non-trivial that for almost every $\omega$ and for each $(s,x)$, the map $t\mapsto X^{s,x}_t(\omega)$ satisfies the equation \eqref{eqn.ODEomega}.
        The main difficulty is to show that the continuous extension of the map
        \begin{align*}
            (s,t,x)\mapsto \int_s^t b(r,\tilde X^{s,x}_r(\omega))\dr= \int_s^t b(r, X^{s,x}_r(\omega))\dr
        \end{align*}
        which is defined on $G$, is identical to the map
        \begin{align*}
            (s,t,x)\mapsto \int_s^t b(r, X^{s,x}_r(\omega))\dr,
        \end{align*}
        which is defined on $\simt\times\Rd$.

        Let $H$ be a closed ball in $\Rd$ and put $\abz=[0,T]\times H$.  
        Let $\tilde{\Omega}_{b,|b|,T}\coloneqq \hat{\Omega}_{b,T}\cap \bigcap_{n \in \mathbb{N}} \Omega^\prime_{|b|,T,n}$, for $\Omega^\prime_{|b|,T,n}$ as in \cref{lem.nonlinYoung}, i.e. \eqref{est.young} and \eqref{est.intfome} hold for $f=|b|$ and any function $\psi$ of finite variation with values in $B_n$. We verify the conditions \ref{it.var} and \ref{it.equicont} of \cref{cor.intcont} for $\omega\in \tilde{\Omega}_{b,|b|,T}$. 
        For each $(t,s,x)\in[0,T]\times\abz$, define $\psi^{s,x}_t(\omega): = X^{s,x}_{\min(s,t)}(\omega)-W_{\min(s,t)}(\omega)$,
         and if $s,x$ are dyadic, define  $\tilde\psi^{s,x}_t(\omega) \coloneqq\tilde X^{s,x}_{\min(s,t)}(\omega)-W_{\min(s,t)}(\omega)$.
        Note that $\psi(\omega)$ is the continuous extension of
    $\tilde \psi(\omega)$. By continuity, there exists $N \in \mathbb{N}$ such that $\psi^{s,x}_\cdot(\omega)$ takes values in $B_N$ for all $(s,x) \in \abz$.
        As \eqref{est.intfome} holds, we can apply \cref{prop.apri1} for the ball $B_N$ to find a constant $C_{T,N}(\omega)$  such that whenever $s,x$ are dyadic,
        \begin{align*}
              [\psi^{s,x}(\omega)]_{\var;[0,T]}=[\tilde\psi^{s,x}(\omega)]_{\var;[0,T]}\le  \int_s^T|b|(r,\tilde X^{s,x}_r(\omega))\dr\le C_{T,N}(\omega) .
        \end{align*}
        This means that for any partition $\pi$ of $[0,T]$ and $(s,x)$ dyadic in $\abz$
        \begin{align*}
            \sum_{[u,v]\in \pi}|\psi^{s,x}_v(\omega)- \psi^{s,x}_u(\omega)|\le C_{T,N}(\omega).
        \end{align*}
        By continuity, the above estimate also holds for every $(s,x)\in \abz$. This implies that
        \begin{align*}
            \sup_{(s,x)\in \abz} [\psi^{s,x}(\omega)]_{\var;[0,T]}\le C_{T,N}(\omega),
        \end{align*}
        verifying \ref{it.var}.
        Condition \ref{it.equicont} is satisfied because $ X(\omega)$ is  locally H\"older continuous on $\simt\times \Rd$.
        Let $\Omega_{b,T}$ be the event in \cref{cor.intcont}. 
        Let $\omega\in \hat{\Omega}_{b,T}\cap \Omega_{b,T}$. Applying \cref{cor.intcont}, we see that the map
            \begin{align*}
                (t,s,x)\mapsto \int_0^t b(r,W_r(\omega)+\psi^{s,x}_r(\omega))\dr=\int_0^t b(r,X^{s,x}_r(\omega))\dr
            \end{align*}
        is continuous on $[0,T]\times\abz$.

        It is now clear from \eqref{tmp.eqn-initial} that for every $\omega\in \tilde{\Omega}_{b,|b|,T}\cap \Omega_{b,T}$, for every $(s,t,x)\in\simt\times H$, we have
        \begin{align*}
             X^{s,x}_t(\omega)=x+\int_s^tb(r, X^{s,x}_r(\omega))\dd r+W_t(\omega)-W_s(\omega) .
        \end{align*}
        By exhausting $\Rd$ with a sequence of increasing closed balls $(H^i)$, we see that the above equation holds for every $(s,t,x)\in\simt\times\Rd$.
        By pathwise uniqueness and using that existence of a strong solution also holds for random initial condition, we have for every $s\le u\le t$ and every $x\in\Rd$
        \begin{align}\label{tmp.contsemiflow}
             {X}^{s,x}_t= {X}^{u, {X}^{s,x}_u}_t \quad \text{a.s.}
        \end{align}
        Note that the exceptional null event \eqref{tmp.contsemiflow} depends on $s,u,t,x$. 
        However, because all processes in \eqref{tmp.contsemiflow} are continuous, one can deduce that on a set $\Omega_{b,T}^{flow}$ of full measure,
        ${X}^{s,x}_t= {X}^{u, {X}^{s,x}_u}_t$ for any $s,u,t,x$, which means that
        $(X^{s,x}_t(\omega))_{s,t,x}$ is a semiflow. 
    \end{proof}
     \begin{remark}\label{rmk.issue}
            The fact that for almost every $\omega$, for every $s,x$, $(X^{s,x}_t(\omega))_t$ is a solution to \eqref{eqn.ODEomega} is crucial for the proof of path-by-path uniqueness. When $b$ is not continuous, this property becomes highly non-trivial and its justification seems missing from the literature (see \cite{AVS,shaposhnikov2017correction}). We achieved this property in \cref{lem.semiflow} by utilizing the regularizing estimates from \cref{lem.nonlinYoung}.
            This issue is irrelevant to \cite{Davie} because of its different arguments.
            We take this chance to note that the exceptional null events in \cite{Davie} depend on the initial condition. This dependence can be removed following our arguments herein. In fact, one just replaces \cref{prop.davie} by Davie's basic estimate, then the rest of the arguments follows with minimal adjustments. In particular, \cref{lem.semiflow,thm.main} below hold with $p=q=\infty$.
        \end{remark}
    \begin{definition}\label{def.Omegabtk}
        For each $T>0$ and $R>0$, let $\Omega_{b,T,R}\coloneqq \Omega^\prime_{b,T,R}\cap \Omega^\prime_{|b|,T,R}\cap \Omega^{flow}_{b,T}$, i.e. the event such that for every $\omega\in
        \Omega_{b,T,R}$,
        \begin{itemize} \item the semiflow $(X^{s,
        x}_t(\omega))$ supplied by \cref{lem.semiflow} is continuous and locally $\kappa$-H\"older continuous in
        space for every $\kappa\in(0,1)$;
         \item \eqref{est.young} and \eqref{est.intfome} hold with
        $f\in\{b,|b|\}$ and any function $\psi$ of finite variation taking
        values in $B_R$. 
    \end{itemize}
    \end{definition}
    In view of \cref{lem.nonlinYoung,lem.semiflow}, $\Omega_{b,T,R}$ has full probability measure. 
    We note that at the current state, for each $\omega\in \Omega_{b,T,R}$,  the semiflow $(X^{s,x}_t(\omega))$ may depend on $T$. This dependence will be removed once we show path-by-path uniqueness.  
	The next result, together with \cref{prop.apri1}, provides \emph{a priori} estimates for any solution to \eqref{eqn.ODEomega} for each $\omega\in \Omega_{b,T,R}$.

 \begin{proposition}\label{prop.regularityabs}
		Let $T>0$ and  $R>0$.
		Let $\omega \in \Omega_{b,T,R}$ and $Y$ be a solution to \eqref{eqn.ODEomega} on $[S,T]$ for some $S \in [0,T]$. 
		Suppose that  $X^{s,Y_s}_t(\omega)-W_t(\omega)$ belongs to $B_R$ for every $(s,t)\in\simst$.
		Then, there exist 
		\begin{itemize}
		 	\item[(i)] a control $w$ which depends only on $Y,\Xi_{T,R}(b)(\omega),\Xi_{T,R}(|b|)(\omega),d,p,q$; 
		 	\item[(ii)] a constant $\beta>1$ which depends only on $d,p,q$
		\end{itemize} 
		such that
		\begin{align}\label{est.apriY}
			|Y_t-X^{s,Y_s}_t(\omega)|\le w(s,t)^{\beta} \quad \forall (s,t)\in\simpl{S}{T}.
		\end{align}
	\end{proposition}
	\begin{proof}
		We omit the dependence of $\omega$ in the argument below.	
		Define $\psi_\cdot=Y_{\min(S,\cdot)}-W_{\min(S,\cdot)}$ and $\xi^{s,x}_\cdot=X^{s,x}_{\min(s,\cdot)}-W_{\min(s,\cdot)}$. The proof  consists of two steps.

		\textit{Step 1.} We show that for $\varepsilon\in(0,\alpha)$, there exist a constant  $C=C(\Xi_{T,R}(|b|),\varepsilon)$ such that 
		\begin{multline}\label{est.absb}
			\Big|\int_s^t[ |b|(r,X^{s,Y_s}_r)- |b|(r,Y_r)]\dd r\Big|\le C \left(\int_s^t|b|(r,Y_r)\dd r\right)^{1- \varepsilon}(t-s)^{\alpha}
			\\+C(t-s)^{\frac \alpha \varepsilon}  \quad \forall (s,t)\in\simst.
		\end{multline}
		As $\psi_s=\xi_s^{s,Y_s}$, we obtain that
		\begin{align*}
			&\int_s^t [|b|(r,X^{s,Y_s}_r)-|b|(r,Y_r)]\dd r
			\\&= \int_s^t[ |b|(r,X^{s,Y_s}_r)-|b|(r,W_r+\xi^{s,Y_s}_s)] \dd r -\int_s^t [|b|(r,Y_r)-|b|(r,W_r+\psi_s)]\dd r.
		\end{align*}
		We note that $\xi^{s,Y_s}$ and $\psi$ are contained in $B_R$ and have finite variation such that
		\begin{gather*}
			[\xi^{s,Y_s}]_{\var;[s,t]}\le \int_s^t|b|(r,X^{s,Y_s}_r)\dd r
			\tand [\psi]_{\var;[s,t]}\le \int_s^t|b|(r,Y_r)\dd r.
		\end{gather*}
		Applying \cref{lem.nonlinYoung}, we have for every $\varepsilon\in(0,\alpha)$ and every $(s,t)\in\simst$,
		\begin{gather*}
			\Big| \int_s^t[ |b|(r,X^{s,Y_s}_r)-|b|(r,W_r+\xi^{s,Y_s}_s)] \dd r\Big|\les \Xi_{T,R}(|b|)\left(\int_s^t|b|(r,X^{s,Y_s}_r)\dd r\right) ^{1- \varepsilon}(t-s)^{\alpha},
			\\\Big| \int_s^t[ |b|(r,Y_r)-|b|(r,W_r+\psi_s)] \dd r\Big|\les \Xi_{T,R}(|b|)\left(\int_s^t|b|(r,Y_r)\dd r\right) ^{1- \varepsilon}(t-s)^{\alpha}.
		\end{gather*}
		We also have
		\begin{align} \label{eq:triangle}
			\Big|\int_s^t|b|(r,X^{s,Y_s}_r)\dd r\Big|^{1- \varepsilon}\le \Big|\int_s^t[|b|(r,X^{s,Y_s}_r)-|b|(r,Y_r)]\dd r\Big|^{1- \varepsilon}+\Big|\int_s^t|b|(r,Y_r)\dd r\Big|^{1- \varepsilon}.
		\end{align}
		It follows that
		\begin{multline*}
			\Big|\int_s^t[ |b|(r,X^{s,Y_s}_r)- |b|(r,Y_r)]\dd r\Big|
			\les \Xi_{T,R}(|b|) \Big|\int_s^t[|b|(r,X^{s,Y_s}_r)-|b|(r,Y_r)]\dd r\Big|^{1- \varepsilon}(t-s)^\alpha
            \\+{\Xi_{T,R}(|b|)}\left(\int_s^t|b|(r,Y_r)\dd r\right)^{1- \varepsilon}(t-s)^\alpha.
		\end{multline*}
		Applying Young's inequality, we have for every $\varepsilon'>0$
		\begin{align}\label{ieq:Young}
			\Xi_{T,R}(|b|)& \Big|\int_s^t[|b|(r,X^{s,Y_s}_r)-|b|(r,Y_r)]\dd r\Big|^{1- \varepsilon}(t-s)^\alpha
			\nonumber\\&\le \varepsilon' \Big|\int_s^t[|b|(r,X^{s,Y_s}_r)-|b|(r,Y_r)]\dd r\Big|+C_{\varepsilon'}\big(\Xi_{T,R}(|b|)(t-s)^\alpha\big)^{\frac1 \varepsilon}.
		\end{align}
		Combining the previous two inequality and choosing $\varepsilon'$ sufficiently small yields \eqref{est.absb}.

		\textit{Step 2.} We show \eqref{est.apriY}.	Let $(s,t)\in\simst$ be fixed.
		Using the equations and the identity $\psi_s=\xi_s^{s,Y_s}$, we obtain that
		\begin{align*}
		 	|X^{s,Y_s}_t-Y_t|
		 	&=\left|\int_s^t [b(r,X^{s,Y_s}_r)-b(r,Y_r)]\dd r\right|
			\\&\leqslant \left|\int_s^t[ b(r,X^{s,Y_s}_r)-b(r,W_r+\xi^{s,Y_s}_s)] \dd r\right| + \left|\int_s^t [b(r,Y_r)-b(r,W_r+\psi_s)]\dd r\right|.
		\end{align*}
		Applying \cref{lem.nonlinYoung}, we have for every $\varepsilon\in(0,\alpha)$,
            \begin{gather*}
			\Big| \int_s^t[ b(r,X^{s,Y_s}_r)-b(r,W_r+\xi^{s,Y_s}_s)] \dd r\Big|\les \Xi_{T,R}(b)\left(\int_s^t|b|(r,X^{s,Y_s}_r)\dd r\right) ^{1- \varepsilon}(t-s)^{\alpha}
			,\\\Big| \int_s^t[ b(r,Y_r)-b(r,W_r+\psi_s)]\dd r\Big|\les \Xi_{T,R}(b)\left(\int_s^t|b|(r,Y_r)\dd r\right) ^{1- \varepsilon}(t-s)^{\alpha}.
		\end{gather*}
		Combining with \eqref{eq:triangle}, we have
		\begin{align*} %
		 	|X^{s,Y_s}_t-Y_t|
			\les \Xi_{T,R}(b) &\Big|\int_s^t[|b|(r,X^{s,Y_s}_r)-|b|(r,Y_r)]\dd r\Big|^{1- \varepsilon}(t-s)^\alpha \\
	        &+\Xi_{T,R}(b)\left(\int_s^t|b|(r,Y_r)\dd r\right)^{1- \varepsilon}(t-s)^\alpha.
		\end{align*}
		We apply \eqref{ieq:Young} and \eqref{est.absb} to get
        \begin{align*}
        |X_t^{s,Y_s}-Y_t|\leqslant C(\varepsilon,\Xi_{T,R}(|b|),\Xi_{T,R}(b))\left(\left(\int_s^t |b|(r,Y_r) \dd r\right)^{1-\varepsilon}(t-s)^\alpha + (t-s)^{\alpha/\varepsilon}\right).
        \end{align*}
		Choosing $\varepsilon$ small enough such that $\alpha/\varepsilon>1$ yields \eqref{est.apriY}.
	\end{proof}
	\begin{proof}[\textbf{Proof of \cref{thm.main}}] The proof is divided into several steps.

        \textit{Step 1.} 
        Let $T>0$ be fixed but arbitrary. 
        For each positive integer $n$, 
		let 
        $\Omega_{b,T,n}$ be the event defined by \cref{def.Omegabtk}.
        We define $$\Omega_{b,T}=\cup_{k=1}^\infty\cap_{n=k}^\infty \Omega_{b,T,n}$$ and note that $\Omega_{b,T}$ has full probability because each $\Omega_{b,T,n}$ does.
        We show below that for each $\omega\in \Omega_{b,T}$, the equation \eqref{eqn.ODEomega} has unique solution on $[S,T]$ for any $S \in [0,T]$.

		Let $\omega \in \cap_{n=k}^\infty \Omega_{b,T,n}$ for some $k$ and let $Y$ be a solution to \eqref{eqn.ODEomega} on $[S,T]$. By continuity, we can
        choose $N\ge k$ such that $B_N$ contains $Y_t$ and $X^{s,Y_s}_t(\omega)-W_t(\omega)$ for every $(s,t)\in\simpl{S}{T}$.
		Let $w$ be the control and $\beta$ be the constant found in \cref{prop.regularityabs}. 
		We note that $w$ depends on $N$. 
		Let $S \in (0,T]$ be a fixed but arbitrary number and define
		\begin{align*}
		 	F(t)=X^{t,Y_t}_\tau, \quad t\in[S,\tau].
		\end{align*}
		For $(s,t) \in \simpl{S}{\tau}$, we obtain by the semiflow property that
		\begin{align*}
			F(t)-F(s)=X^{t,Y_t}_\tau-X^{s,Y_s}_\tau
			=X^{t,Y_t}_\tau-X^{t,X^{s,Y_s}_t}_\tau.
		\end{align*}
		Then, using H\"older continuity of the semiflow (\cref{lem.semiflow}), we have
		\begin{align*}
			|F(t)-F(s)|\les |Y_t-X^{s,Y_s}_t|^\kappa,
		\end{align*}
		for some $\kappa\in(0,1)$ which can be chosen to be arbitrarily close to $1$. Applying \eqref{est.apriY}, we have
		\begin{align*}
			|F(t)-F(s)|\les w(s,t)^{\kappa \beta}.
		\end{align*}
        Choosing $\kappa$ so that $\kappa \beta>1$, for any sequence of partitions $\Pi_n=\{t^n_i\}_{i=0}^{N_n}$ of $[S,\tau]$ with mesh  converging to $0$, we get that
        \begin{align*}
            |F(\tau)-F(S)|\leqslant \sum_{i=1}^{N_n} |F(t^n_{i})-F(t^n_{i-1})|\les |\sup_{i\le N_n}w(t^n_{i-1},t^n_i) |^{\kappa \beta-1}w(0,\tau)\overset{n\rightarrow \infty}{\longrightarrow} 0.
        \end{align*}
        Hence $F(\tau)=F(S)$ or equivalently $Y_\tau=X^{S,Y_S}_\tau$.  
        Since $\tau$ was arbitrarily chosen in $(S,T]$, this means that $Y$ and $X^{S,Y_S}$ are identical on $[S,T]$.

        \textit{Step 2. Proof of part \ref{it.complete}}. 
        Define
        \begin{align*}
            \Omega_b=\cap_{T=1}^\infty \Omega_{b,T}.
        \end{align*}
        Note that $\Omega_b$ has full probability measure because each $\Omega_{b,T}$ does.  
        Let $(X^{s,x}_t)$ and $(\bar X^{s,x}_t)$ be random continuous, locally H\"older continuous in space semiflows
        respectively over $[0,T]$ and $[0,\bar T]$ for some $T<\bar T$. Then, for every $\omega\in \Omega_b$,  by
        continuity and Step 1, $\bar X|_{\simt\times\Rd}(\omega)=X(\omega)$.
        Hence, the semiflow can be uniquely extended to the whole positive real line, showing part \ref{it.complete}.

        \textit{Step 3.} Part \ref{it.pbpunique} is a direct consequence of steps 1 and 2.
        \end{proof}

          \begin{proof}[\textbf{Proof of \cref{cor.strcomplete}}]
        Let $(X^{s,x}_t)$ and $(\bar X^{s,x}_t)$ be random continuous semiflows respectively over $[0,T]$ and $[0,\bar T]$ for some $T<\bar T$. Then by path-by-path uniqueness $\bar X|_{\simt\times\Rd}=X$. This implies the existence (and uniqueness) of a random continuous semiflow over all nonnegative time.
    \end{proof}

\section{Wider implications} \label{sec:implications}

    This section is dedicated to consequences of \cref{lem.semiflow} and \cref{thm.main}. %

\subsection{Inverse flow and the solution to the backward in time SDE}
In view of \cref{rem:flow}, we can find an event $\Omega_b$ of full probability such that the conclusions of  \cref{thm.main} hold and for each $\omega\in\Omega_b$, the $W(\omega)$-driven ODE \eqref{eqn.ODEomega} generates a flow of homeomorphisms $(X_t^{(s,\cdot)}(\omega))_{s,t}$. 
We denote the inverse of $x\mapsto X_t^{s,x}(\omega)$ by $(X_t^{s,x}(\omega))^{-1}$.
For each fixed $\tau\in[0,1]$ and $x\in \Rd$, we consider the   $W(\omega)$-driven backward-in-time equation 
\begin{align}
     \label{backward-sde}
   	Z^{\tau,x} _{s}&=x-\int_s^\tau b(r,Z^{\tau,x}_{r})dr-(W_\tau- W_s)(\omega),\quad 0\leq s\leq \tau. 
\end{align}
Then, the inverse flow $s\mapsto(X_\tau^{s,x} )^{-1}$ and $s\mapsto Z^{\tau,x}_s$ coincide. Indeed, when $b$ is a regular function, the result is classical, see \cite[Chapters 3,4]{MR1472487}. The case when $b$ satisfies condition \eqref{LpLq} was shown in \cite{FF} through a transformed SDE resulting from a Zvonkin-type transformation.
Here, we provide a different argument for this relation under \eqref{LpLq}, which is valid for each $\omega\in \Omega_b$
In contrast to previous arguments, our proof relies solely on path-by-path uniqueness and the regularity of the flow of \eqref{eqn.ODEomega}.
\begin{theorem}
    \label{cor:flow-back}  Let $b$ fulfill \eqref{LpLq} for all $T>0$ and define $\Omega_b$ as previously. Then for every $(\tau,x,\omega)\in [0,\infty)\times\mR^d\times\Omega_{b}$, Equation \eqref{backward-sde} has a unique solution 
    given by $Z_s^{\tau,x}=(X_\tau^{s,x}(\omega))^{-1}$ for all $s\in[0, \tau]$.
\end{theorem}
\begin{proof}
    Let $\omega \in \Omega_b$. For each $s\le t\le \tau$, any solution $Z^{\tau,x}$ of \eqref{backward-sde} satisfies
    \begin{align*}
        Z^{\tau,x}_t-Z^{\tau,x}_s=\int_s^tb(r,Z^{\tau,x}_r)\dd r+W_t(\omega)-W_s(\omega),
    \end{align*}
    which implies that $Z^{\tau,x}$ is a solution to \eqref{eqn.ODEomega} starting at $Z^{\tau,x}_s$ at time $s$. Hence, by path-by-path uniqueness (\cref{thm.main}), we have $X^{s,Z^{\tau,x}_s}_t(\omega) =Z^{\tau,x}_t$ for all $t\in[s,\tau]$. For $t=\tau$, this yields $X^{s,Z^{\tau,x}_s}_\tau(\omega) =x$, which implies the result. 
\end{proof}

  \subsection{Random dynamical systems generated by solutions of singular  SDEs}\label{RDS:path-by-path}
   In general, to study systems that evolve over time under the influence of randomness—modeled by SDEs—one typically "lifts" the solution of an SDE into the framework of a random dynamical system (RDS). This lifting is commonly achieved by interpreting the SDE as a well-posed random ODE, such as \eqref{eqn.ODEomega}, in a path-by-path manner. For Lipschitz continuous drift
$b$, this approach is standard and well-documented (see \cite[Chapter 2.2]{MR1723992}). As discussed earlier and established in \cref{thm.main,cor:flow-back}, when 
$b$ satisfies condition \eqref{LpLq}, we can still ensure the path-by-path well-posedness of \eqref{eqn.ODEomega} in both directions of time. Consequently, under this assumption, the existence of an RDS induced by \eqref{eqn.ODEomega} follows directly via a standard argument which associates such random ODE to an RDS, without requiring further elaboration (cf. \cite[Theorem 2.2.2]{MR1723992}).

\subsection{SDEs with boundary conditions}\label{subsec:boundary}
              The following \cref{lem:BVproblem} generalizes results in \cite{NualartPardoux} to boundary value problems for an SDE with a drift that is possibly discontinuous (in contrast to a local Lipschitz condition required therein). We establish well-posedness for a class of such problems in \cref{lem:BVproblem}. The part on uniqueness follows along the same lines as the proof in \cite{NualartPardoux}; the proof of existence of a solution however follows easily due \cref{thm.main} by a fixed point theorem. We lay out the setting and reformulation of the problem below.
        
        For $\tau \in [0,T]$, we consider the boundary value problem
        \begin{align}\label{eq.BCsde}\left\{\begin{array}{cc}
                 \dd X_t =b(t,X_t)\dd t +\dd W_t,&\\
               F_0 X_0+F_1   X_\tau=c &
             \end{array}\right.
        \end{align}
        under the following assumption.
        \begin{assumption}\label{ass:BV}
        Let $b$  fulfill \eqref{LpLq} such that $|b|\1_{\Rd\setminus B_R }$ belongs to $\LL^1_\infty([0,T])$ for some $R>0$.
        Let $F_0,F_1$ be bounded linear mappings from $\mathbb{R}^d$ to $\mathbb{R}^d$ such that $F\coloneqq F_0+F_1$ has a bounded inverse and $c \in \mathbb{R}^d$.
        \end{assumption}
        \begin{theorem}
             \label{lem:BVproblem}\phantom{new line}
            \begin{enumerate}[(a)]
            \item \label{en:existence}Let \cref{ass:BV} hold. Then there exists a solution on $\Omega_b$ to \eqref{eq.BCsde} for any $\tau \in [0,T]$.
        \item \label{en:uniqueness}Let  $g(z)\coloneqq F^{-1}c-F^{-1}F_1z$. 
        If there exists $\lambda \in \mathbb{R}$ such that
        $b(t,\cdot)+\lambda \mathbbm{1}$ is monotone\footnote{A mapping $h:\mR^d\rightarrow\mR^d$ is said to be monotone if $\big(h(x)-h(y),x-y\big) \geq0$ for any $x,y\in\mR^d$.} for a.e. $t$ and
        \begin{align}\label{eq:ass-uniq}
            e^{\lambda\tau} &|g(z)-g(z')|\geqslant |z-z'+g(z)-g(z')|\implies g(z)=g(z');
        \end{align}
        or if both $b$ and $g$ are monotone, then (path-by-path) uniqueness holds.
        \end{enumerate}
        \end{theorem}
        \begin{proof}[Proof of \cref{lem:BVproblem}]
            Throughout the proof let $\omega \in \Omega_{b}$ for $\Omega_b$ as in \cref{thm.main}. 
            Because the entire argument is with this fixed $\omega$, hence we omit it in the notations.
            
            \ref{en:existence}: If $X$ is a solution to \eqref{eq.BCsde}, by path-by-path uniqueness $X_\cdot=X^{0,X_0}_{\cdot}$. Hence, showing existence of a solution is equivalent to showing that there exists $X_0 \in \mathbb{R}^d$ such that
            \begin{align}\label{eq:flowreformulation}
                F_0 X_0 +F_1 X^{0,X_0}_{\tau}=c.
            \end{align}
            Note that the above is equivalent to
            \begin{align}\label{eq:equiv}
                X_0=F^{-1}c-F^{-1}F_1(X_{\tau}^{0,X_0}-X_0).%
            \end{align}
            Hence, to find a solution to \eqref{eq.BCsde}, it is sufficient to prove that 
            \begin{align}\label{eq:f}
            f(x)=F^{-1}c-F^{-1}F_1(X^{0,x}_\tau-x)
            \end{align}
            has a fixed point.
         To do so, we apply Schauder's fixed point theorem; 
         which requires that $f$ is continuous compact (mapping bounded sets to relatively compact sets) and that the set 
         \begin{align}\label{eq:boundedset}
        \{x: x=\rho f(x) \text{ for some } \rho \in (0,1]\}
        \end{align}
        is bounded.
        Continuity of $f$ follows immediately from continuity of the flow (see \cref{thm.main}\ref{it.complete}).
        To show that $f$ is compact and the set in \eqref{eq:boundedset} is bounded, it suffices to show that $f$ is bounded on $\Rd$.
        
        To see this, putting $b^1=b\1_{B_R}$ and $b^2=b\1_{\Rd\setminus B_R}$, we consider $x$ with 
        \begin{align}\label{ass:x}
            |x|>\bar R:= 1+R+\tau \|b^2\|_{\LL^1_\infty([0,T])} +\sup_{s \in [0,\tau]}|W_s|,
        \end{align}
        where $R>0$ is as in \cref{ass:BV}. 
        This ensures that $X_r^{0,x} \notin B_R$ for all $r \in [0,\tau]$. Indeed, let $\tau_R$ be the hitting time of $B_R$ for $X_\cdot^{0,x}$ and assume that $\tau_R\le \tau$. By continuity, we have $|X^{0,x}_{\tau_R}|=R$. On the other hand,  for $r< \tau_R$, we have
        \begin{align*}
            |X_r^{0,x}|&=\Big|x +\int_0^r b^2(s,X_s^{0,x}) \dd s + W_r\Big|\\
            &\geqslant |x|-\Big|\int_0^r b^2(s,X_s^{0,x})\dd s\Big|-|W_r|\ge R+1
        \end{align*}
        where we used \eqref{ass:x} in the last line. This implies that $|X_{\tau_R}^{0,x}|>R$, a contradiction. Thus, one has $\tau_R>\tau$, in other words, $X_r^{0,x} \notin B_R$ for all $r \in [0,\tau]$.
        It follows that
        \begin{align}\label{eq:boundingf}
            |f(x)|&\leqslant |F^{-1}c|+\|F^{-1}F_1\||X_\tau^{0,x}-x|\nonumber\\
            &\leqslant |F^{-1}c|+\|F^{-1}F_1\|\left(\int_0^\tau |b^2|(r,X^{0,x}_r) \dr + |W_\tau|\right)\nonumber\\
            &\leqslant |F^{-1}c|+\|F^{-1}F_1\|\left( \|b^2\|_{\LL^1_\infty}  + |W_\tau|\right),
        \end{align}
        which is valid for all $x$ satisfying \eqref{ass:x},
        implying that $f(\Rd\setminus B_{\bar R})$ is bounded.
        It remains to show that $f(B_{\bar R})$ is bounded.
        Let $x \in B_{\bar R}$ and let $\sigma_{\bar R}\coloneqq \sup\{s \in [0,\tau]:|X_s^{0,x}|\leqslant \bar R\}$. 
        We have $|X_{\sigma_{\bar R}}^{0,x}|=\bar R$ (by continuity) and  $|X_r^{0,x}|>\bar R>R$ for all $r\in(\sigma_{\bar R},\tau]$. Hence, using the identity $X_\tau^{0,x}=X_{\sigma_{\bar R}}^{0,x}+\int_{\sigma_{\bar R}}^\tau b^2(r,X_r^{0,x})\dr+W_\tau-W_{\sigma_{\bar R}}$ and triangle inequality, we have 
        \begin{align*}
           |f(x)|&\leqslant |F^{-1}c|+\|F^{-1}F_1\|(|X_\tau^{0,x}|+|x|)\\
           &\leqslant |F^{-1}c|+\|F^{-1}F_1\|\Big(|\bar R|+\int_{\sigma_{\bar R}}^\tau |b^2|(r,X_r^{0,x}) \dr +|W_{\tau}-W_{\sigma_{\bar R}}|+|\bar R|\Big)\\
           &\leqslant |F^{-1}c|+\|F^{-1}F_1\|(2|\bar R|+\|b^2\|_{\LL^1_\infty} +2\sup_{s \in [0,\tau]}|W_s|).
        \end{align*}
        This implies that $f(B_{\bar R})$ is bounded. We conclude that $f$ is bounded on $\Rd$, completing the proof.
        
        \ref{en:uniqueness}: By path-by-path uniqueness (\cref{thm.main}), any solution to \eqref{eq.BCsde} is of the form $X^{0,x}_\cdot$ for some $x$. 
        Let $x,\tilde{x}$ be such that $X^{0,x}_\cdot$ and $X^{0,\tilde{x}}_\cdot$ fulfill \eqref{eq:flowreformulation}.
        Let $\lambda$ be such that the monotonicity condition and \eqref{eq:ass-uniq} holds. Then by Lebesgue differentiation theorem, for a.e. $t$, we have
        \begin{align*}
            \frac{d}{dt}(&e^{-2\lambda t} \big|X_t^{0,x}-X_t^{0,\tilde{x}}\big|^2)\\
            &+2 e^{-2\lambda t}\left( \lambda\big|X_t^{0,x}-X_t^{0,\tilde{x}}\big|^2 + \langle b(t,X_t^{0,x})-b(t,X_t^{0,\tilde{x}}),X_t^{0,x}-\tilde{X}_t^{0,\tilde{x}}\rangle\right)=0.
        \end{align*}
        Therefore by the monotonicity assumption on $b(t,\cdot)+\lambda \mathbbm{1}$,
        \begin{align*}
            \frac{d}{dt}(e^{-2\lambda t} \big|X_t^{0,x}-X_t^{0,\tilde{x}}\big|^2)\leqslant 0
        \end{align*}
        giving that
        \begin{align} \label{eq:gz}
            |X_\tau^{0,x}-X_\tau^{0,\tilde{x}}|\leqslant e^{\lambda \tau} |x-\tilde{x}|.
        \end{align}
        As $X^{0,x}$ and $X^{0,\tilde{x}}$ are solutions, by \eqref{eq:equiv}, we have that
        $g(X_\tau^{0,x}-x)=x$ and $g(X_\tau^{0,\tilde{x}}-\tilde{x})=\tilde{x}$, so that the assumption in \eqref{eq:ass-uniq} is fulfilled for $z=X^{0,x}_\tau-x$ and $z'=X^{0,\tilde{x}}_\tau-\tilde{x}$. This implies that $x=\tilde{x}$. The proof of uniqueness when $b$ and $g$ are monotone follows along similar lines (see \cite[Remark 2.4]{NualartPardoux}).
        \end{proof}

\begin{appendix}
\section{ODE uniqueness} %
\label{sec:ode_uniqueness}
    Most of the arguments in \cref{sec:proof_main_results} are independent from the probability space, and hence, independent from the probability law of the driving noise. This forms a uniqueness criterion for ODE  \eqref{ODE} which does not require any regularity on the vector field, but instead relies on the regularizing effect of the driving signal $\gamma$ with \cref{lem.nonlinYoung} being a prototype. This regularizing effect is formalized by the following definition.
    \begin{definition}\label{def.gregularizing}
        Let $f:[0,T]\times\Rd\to\Rd$ be a measurable function. We say that $\gamma$ is $(1- \varepsilon, \alpha)$-regularizing for $f$ if there exist a control $\eta$ and constants $\Xi_{T,R}$ for each $R>0$ such that 
        \begin{align*}
            \Big|\int_s^t [f(r,\gamma_r+ \psi_r)- f(r,\gamma_r+ \psi_s)]\dd r\Big|\le \Xi_{T,R} [\psi]_{\var;[s,t]}^{1- \varepsilon} \eta(s,t)^{\alpha}
        \end{align*}
        for every $(s,t)\in\simt$ and every $\psi:[0,T]\to B_R$ of finite variation.
    \end{definition}
        
    \begin{theorem}\label{thm.ODEuniq}
        Let $R>0$. 
        Suppose that $\gamma$ is $(1- \varepsilon, \alpha)$-regularizing for $b,|b|$. 
        Let $y$ be a solution to \eqref{ODE}. 
        Suppose that $ \phi^{s,y_s}_t-\gamma_t$ belongs to $B_R$ for every $(s,t)\in\simt$. 
        Then, there exists  a control $w$ which depends only on $y,\Xi_{T,R}, \eta$
        such that
        \begin{align}\label{est.apriY2}
            |y_t- \phi^{s,y_s}_t|\le w(s,t)^{\min(1- \varepsilon+ \alpha, \frac \alpha \varepsilon)} \quad \forall (s,t)\in\simt.
        \end{align}
        Suppose furthermore that $\phi$ is locally $\kappa$-H\"older continuous for some $\kappa\in(0,1]$ such that
        \begin{align*}
            \kappa \cdot\min(1- \varepsilon+ \alpha, \frac \alpha \varepsilon)>1.
        \end{align*}
        Then   $(y_t)_{t\in[0,T]}$  is identical to $(\phi^{0,y_0}_t)_{t\in[0,T]}$.    
    \end{theorem}
    The proof of this result follows analogous arguments used in proving \cref{prop.regularityabs,thm.main} and hence is omitted. 
 
\section{Adapted solutions and path-by-path solutions}\label{pbpvspathwise}
This section is devoted to summarize different concepts of solutions and uniqueness to \eqref{SDE} where $b: [0,T]\times \mathbb{R}^d \rightarrow \mathbb{R}^d$ is measurable. Even though the concept of adapted solutions is by now well covered in textbooks, we also recall these notions in order to make a clear comparison to path-by-path solutions and path-by-path uniqueness (for a definition thereof see \cref{def.unipbp}. In particular, we state and prove \cref{thm:pbp} to point out the connection between path-by-path uniqueness and adaptedness of solutions.

\begin{definition}[Existence]\phantom{newline}
\begin{enumerate}[(i)]
\item If there exists a filtered probability 
space $(\Omega,\mathcal{F},(\mathcal{F}_t)_{t \in [0,T]},\mathbb{P})$ equipped with a Brownian motion $(W_t)_{t \in [0,T]}$ and an $(\mathcal{F}_t)$-adapted process $(X_t)_{t \in [0,T]}$ such that $\int_0^T |b(s,X_s)|\dd s<\infty$ and $(X,W)$ fulfills \eqref{SDE} almost surely, we say that $(X,W)$ is a \emph{weak solution} to \eqref{SDE}. If the choice of $W$ is clear from the context, we write that $X$ is a weak solution.

\item We call $X$ a \emph{strong solution} if $X$ is a weak solution and $(X_t)_{t\in[0,T]}$ is adapted to the filtration $(\mathcal{F}_t^W)_{t\in[0,T]}$.

\item Let $(\Omega,\mathcal{F},\mathbb{P})$ be a probability space on which a Brownian motion $W$ is defined. We call a mapping $X\colon\Omega \rightarrow \mathcal{C}([0,T])$ a \emph{path-by-path solution} to \eqref{SDE} if there exists a set $\tilde{\Omega}\subset \Omega$ of full measure such that, for every $\omega \in \tilde{\Omega}$, $X(\omega)$ is a solution to \eqref{eqn.ODEomega}.
\end{enumerate}

\end{definition}
\begin{definition}[Uniqueness]\phantom{newline}
\begin{enumerate}[(i)]
\item We say that \emph{pathwise uniqueness} for \eqref{SDE} holds if for any two weak solutions $(X,W), (\tilde{X},W)$ defined on the same filtered probability space with the same Brownian motion $W$ and the same initial condition, $X$ and $\tilde{X}$ are indistinguishable.

\end{enumerate}
\end{definition}

The following implications follow directly from the 
definitions:
\begin{center}
\begin{tikzpicture}
\draw (0,-0.3) rectangle (3,0.3) node[pos=.5] {strong existence};
\draw (4,-0.3) rectangle (7,0.3) node[pos=.5] {weak existence};
\draw (8,-0.3) rectangle (12.2,0.3) node[pos=.5] {path-by-path existence};
\draw[->,double equal sign distance] (3.2,0) -- (3.8,0);
\draw[->,double equal sign distance] (7.2,0) -- (7.8,0);
\draw (1.5,-1.3) rectangle (6,-0.7) node[pos=.5] {path-by-path uniqueness};
\draw[->,double equal sign distance] (6.2,-1) -- (6.8,-1);
\draw (7,-1.3) rectangle (10.7,-0.7) node[pos=.5] {pathwise uniqueness};
\end{tikzpicture}
\end{center}
The following theorem can be summarized in the following way: For drifts for which pathwise uniqueness is known to hold, losing path-by-path uniqueness implies existence of a non-adapted solution. Note that the theorem can be formulated with random initial condition if assuming \emph{uniform} path-by-path uniqueness.
 \begin{theorem}\label{thm:pbp}
Consider \eqref{SDE} with deterministic initial condition and $b:[0,T]\times \mathbb{R}^d \rightarrow \mathbb{R}^d$ measurable.
Assume that 
\begin{itemize}
    \item[(i)] pathwise uniqueness to \eqref{SDE} holds,
    \item[(ii)] all path-by-path solutions to \eqref{SDE} are weak solutions.
\end{itemize}
Then path-by-path uniqueness holds.
\end{theorem}
\begin{proof}
Assume that path-by-path uniqueness does not hold. Let $(\Omega,\mathcal{F},\mathbb{P})$ be a probability space on which a Brownian motion $W$ is defined. Then there exists a set $A\subset \Omega$ of positive measure such that for $\omega \in A$, there exist multiple solutions to
\begin{align*}
X_t(\omega)=X_0+\int_0^t b(s,X_s(\omega)) \dd s + W_t(\omega),
\end{align*}
with the same initial condition. Then we can define two path-by-path solutions $X^1$ and $X^2$ to \eqref{SDE} by letting them agree with the unique strong solution on $A^{\mathsf{c}}$ and letting $X^1(\omega)\neq X^2(\omega)$ on $A$, which is possible by the above and the axiom of choice.
By assumption, both $X^1$ and $X^2$ are adapted w.r.t. filtrations $(\mathcal{F}^1_t)$ and $(\mathcal{F}^2_t)$ such that $W$ is a Brownian motion w.r.t. these filtrations. Hence, $W$ is also a Brownian motion with respect to $(\mathcal{F}^1_t \cup \mathcal{F}^2_t)$. Therefore we constructed two solutions on the same filtered probability space $(\Omega,\mathcal{F},(\mathcal{F}^1_t\cup \mathcal{F}^2_t)_{t\in[0,T]},\mathbb{P})$. By construction $X^1\neq X^2$ on a set of positive measure. This contradicts pathwise uniqueness.
\end{proof}

\begin{remark}

There are examples of drifts $b$ not fulfilling \eqref{LpLq} such that there exist multiple (non-adapted) path-by-path solutions to \eqref{SDE}, even though pathwise uniqueness holds (see \cite{ShaWre} for $d>1$ and \cite{Anzeletti} for $d=1$). These counterexamples heavily rely on the time-dependence of the drift. This is leading to non-uniqueness globally (i.e. considering the equation on the whole interval $[0,T]$). We are not aware of examples of drifts for which such a phenomenon occurs locally. Finally, note that \cref{thm.main} prevents such counterexamples for drifts fulfilling \eqref{LpLq} since path-by-path solutions are identified on a set of full measure.
\end{remark}

\end{appendix}

\section*{Acknowledgments}
The authors thank Toyomu Matsuda for numerous suggestions which improve the quality of the paper and  Oleg Butkovsky for suggesting \cref{sec:ode_uniqueness}. Discussion on strong completeness with Prof. Scheutzow is acknowledged.
\section*{Funding} %
LA acknowledges the support of the Labex de Math\'ematique Hadamard and of the Austrian Science Fund (FWF) via program P34992. KL is supported by EPSRC [grant number EP/Y016955/1], CL is supported by the Austrian Science Fund (FWF) Stand-Alone program P34992.  CL also acknowledges the financial support from DFG via Research Unit FOR 2402 when this project started. This research was funded in whole or in part by the Austrian Science Fund (FWF) \href{https://www.fwf.ac.at/en/research-radar/10.55776/P34992}{10.55776/P34992}. For open access purposes, the author has applied a CC BY public copyright license to any author-accepted manuscript version arising from this submission.

\end{document}